\documentclass[10pt,english]{article}

\usepackage{ae,aecompl}
\usepackage[T1]{fontenc}
\usepackage[latin9]{inputenc}
\usepackage{geometry}
\usepackage{babel}
\usepackage{prettyref}
\usepackage{mathrsfs}
\usepackage{amsmath}
\usepackage{amsthm}
\usepackage{amssymb}
\usepackage{graphicx}
\usepackage{setspace}
\usepackage[unicode=true,pdfusetitle,
 bookmarks=true,bookmarksnumbered=true,bookmarksopen=true,
 breaklinks=false,pdfborder={0 0 1},backref=true,colorlinks=true,linkcolor=blue,citecolor=cyan]
 {hyperref}

\makeatletter

\numberwithin{figure}{section}
\numberwithin{equation}{section}
\theoremstyle{remark}
\newtheorem*{rem*}{\protect\remarkname}
\theoremstyle{plain}
\newtheorem{thm}{\protect\theoremname}[section]
\theoremstyle{plain}
\newtheorem{cor}[thm]{\protect\corollaryname}
\theoremstyle{plain}
\newtheorem{prop}[thm]{\protect\propositionname}
\theoremstyle{plain}
\newtheorem{lem}[thm]{\protect\lemmaname}
\theoremstyle{definition}
\newtheorem{defn}[thm]{\protect\definitionname}
\theoremstyle{definition}
\newtheorem{example}[thm]{\protect\examplename}
\theoremstyle{remark}
\newtheorem{notation}[thm]{\protect\notationname}

\usepackage{bbm}

\date{}

\makeatother

\providecommand{\corollaryname}{Corollary}
\providecommand{\definitionname}{Definition}
\providecommand{\examplename}{Example}
\providecommand{\lemmaname}{Lemma}
\providecommand{\notationname}{Notation}
\providecommand{\propositionname}{Proposition}
\providecommand{\remarkname}{Remark}
\providecommand{\theoremname}{Theorem}

\begin{document}

\title{\textbf{Eigenvalues and spectral gap in sparse random simplicial complexes}}

\author{Shaked Leibzirer\thanks{Partially supported by ISF 771/17 and BSF grant 2018330.}
~and~ Ron Rosenthal$^{*}$}

\maketitle


\begin{abstract}
We consider the adjacency operator $A$ of the Linial-Meshulam model
$X(d,n,p)$ for random $d-$dimensional simplicial complexes on $n$
vertices, where each $d-$cell is added independently with probability
$p\in[0,1]$ to the complete $(d-1)$-skeleton. We consider sparse
random matrices $H$, which are generalizations of the centered and
normalized adjacency matrix $\mathcal{A}:=(np(1-p))^{-1/2}\cdot(A-\mathbb{E}\left[A\right])$,
obtained by replacing the Bernoulli$(p)$ random variables used to
construct $A$ with arbitrary bounded distribution $Z$. We obtain
bounds on the expected Schatten norm of $H$, which allow us to prove
results on eigenvalue confinement and in particular that $\left\Vert H\right\Vert _{2}$
converges to $2\sqrt{d}$ both in expectation and $\mathbb{P}-$almost
surely as $n\to\infty$, provided that $\mathrm{Var}(Z)\gg\frac{\log n}{n}$.
The main ingredient in the proof is a generalization of \cite[Theorem 4.8]{LHY18}
to the context of high-dimensional simplicial complexes, which may
be regarded as sparse random matrix models with dependent entries.
\end{abstract}


\section{Introduction}

The Erd\H{o}s--Rényi graph (\cite{ER59,ER61}) $G\left(n,p\right)$,
is a random graph on $n$ vertices, where each edge is added independently
with probability $p\in\left[0,1\right]$ that might depend on $n$.
The model and particularly the spectrum of its adjacency matrix $A$
has been extensively studied. For insteance, it follows from Wigner's
semicircle theorem that the spectrum of $(np(1-p))^{-1/2}(A-\mathbb{E}[A])$
converges weakly in probability to the semicircle law, provided $\lim_{n\to\infty}np(1-p)=\infty$,
and it is shown in \cite{FK81,VH07} that under the assumption $np\gg\log^{4}\left(n\right),$
namely $\lim_{n\to\infty}(np)^{-1}\log^{4}\left(n\right)=0$, one
has
\begin{equation}
\mathbb{E}\left[\frac{\left\Vert A-\mathbb{E}\left[A\right]\right\Vert }{\sqrt{np}}\right]\leq2\left(1+o\left(1\right)\right),\qquad\text{ as \ensuremath{n\to\infty},}\label{eq:1-1}
\end{equation}
which is a well-known result regarding the spectral gap of the adjacency
matrix of the homogeneous Erd\H{o}s--Rényi graph. Recently, it was
shown independently in \cite{MR4116720} and \cite{LHY18} that \eqref{eq:1-1}
holds under the weaker assumption $np(1-p)\gg\log\left(n\right)$,
which by \cite{MR4116720,MR3945756,alt2021extremal} is the optimal
regime.

The Linial-Meshulam model, c.f. \cite{LM06,MW09}, is a high-dimensional
generalization of the Erd\H{o}s--Rényi model. Given $n,d\in\mathbb{N}$
such that $n\geq d+1$, and $p\in\left[0,1\right]$ that might depend
on $n$, the Linial-Meshulam model $X\equiv X\left(d,n,p\right)$,
is a random $d$-dimensional simplicial complex on $n$ vertices with
a complete $(d-1)$-skeleton, in which each $d$-cell is added to
$X$ independently with probability $p$. When $d=1$, the model reduces
to the Erd\H{o}s--Rényi random graph. Since its appearance the Linial-Meshulam
model attracted much attention, see for example \cite{MW09,NK10,BHK11,W11,HJ13,ALLM13,CCFK16,GW16,LP16,RK17,YT17,PR17,MR3639605,ODMP18,LP18,LP19,MR3984242,fountoulakis2020algebraic,LP22}.

Let $A$ be the adjacency operator associated with the Linial-Meshulam
model $X\equiv X(d,n,p)$, see Section \ref{sec:Results} for a precise
definition. In this paper we consider a random matrix $H$ which is
a generalization of the centered and normalized adjacency matrix,
defined by $\mathcal{A}:=\left(np\left(1-p\right)\right)^{-\frac{1}{2}}\left(A-\mathbb{E}\left[A\right]\right)$,
where the Bernoulli$(p)$ random variables used to construct $A$
are replaced with an arbitrary bounded distribution $Z$. The matrix
$H$ is a sparse self-adjoint random matrix equipped with the same
dependent structure as $\mathcal{A}$, and in particular its entries
are only independent (up to the self-adjointness constraint) if and
only if $d=1$ (see \cite{RK17} for further details).

Our main result is a generalization of \eqref{eq:1-1} to random matrices
of type $H$ and in particular to the rescaled and centered adjacency
matrix $\mathcal{A}$ of random simplicial complexes $X(d,n,p)$.
Previous results related to the spectrum of the adjacency matrix $A$
for arbitrary $d\in\mathbb{N}$ were introduced in \cite{RK17} and
\cite{GW16}. In \cite[Theorem 5.1]{RK17}, the authors assume $np(1-p)\gg\log^{4}\left(n\right)$
and at the cost of this stronger assumption (compared to $np(1-p)\gg\log\left(n\right)$)
establish \eqref{eq:1-1}, for all $d\geq2$ with the appropriate
optimal bound in the right hand side of \eqref{eq:1-1}. On the other
hand, in \cite[Theorem 2]{GW16}, it is shown that under the assumption
$np(1-p)\gg\log\left(n\right)$, one can obtain an upper bound on
the left hand side of \eqref{eq:1-1}, which is not optimal. In \cite{LHY18},
a key ingredient in problem the proof of \eqref{eq:1-1} for the Erd\H{o}s--Rényi
model, namely the case $d=1$ is \cite[Theorem 4.8]{LHY18}, which
assumes independent entries (up to self-adjointness). Due to the dependent
structure of the entries of $H$, one can not apply \cite[ Theorem 4.8]{LHY18}
whenever $d>1$. In this paper, we generalize \cite[ Theorem 4.8]{LHY18}
to random matrices of type $H$ (see Theorem \ref{thm:main_result}),
which allows us to obtain bounds on the $2k-$Schatten norm 
\[
\mathbb{E}\left[\left\Vert H\right\Vert _{S_{2k}}\right]:=\mathbb{E}\bigg[\sqrt[2k]{\text{Trace}\left(\left|H\right|^{2k}\right)}\bigg]\,,
\]
for all $d\geq1$, where $\left|H\right|=\sqrt{H^{*}H}$. Furthermore,
we use this bound in order to show that $\lim_{n\to\infty}\mathbb{E}[\left\Vert H\right\Vert _{2}]=2\sqrt{d}$,
provided $\text{Var}\left(Z\right)\gg n^{-1}\log n$, which by \cite{RK17}
is the optimal bound. We thus derive the optimal bound achieved in
\cite{RK17} under weaker assumptions, which coincide with those postulated
in \cite{GW16,LM06,MW09}. In addition to the norm bound, we improve
the bound on $\mathbb{P}(\left\Vert H\right\Vert _{2}>2\sqrt{d}+\varepsilon)$
obtained in \cite[Theorem 5.1]{RK17} for $\varepsilon>0$. We conclude
this section with a few words about the proof of Theorem \ref{thm:main_result}.
The proof is based on the simplicial structure of the entries of $H$.
Te structure allow us to translate the problem into a combinatorial
one by associating a simplicial complex with an embedded path with
each of the elements in the sum defining $\text{Trace}(\left|H\right|^{2k})$.
The simplicial structure brings into play new phenomena regarding
the relation between the path and the simplicial complexes that do
not arise in the graph case, and in particular is not entirely local.
Thus new ideas are required, see Lemma \ref{lem:Lemma1} and Lemma
\ref{Lem:lemma_2} for further details.

\section{Preliminaries and results\label{sec:Results}}

A finite simplicial complex $X$ on a vertex set $V$ is a finite
collection of subsets of $V$ that is closed under taking subsets.
Namely, if $\tau\in X$ and $\sigma\subseteq\tau$, then $\sigma\in X$.
The elements of $X$ are called \emph{cells}, and the dimension of
a cell $\tau$, is defined as $\dim(\tau):=\left|\tau\right|-1$.
For $j\geq-1$, $X^{j}$ denotes the set of cells of dimension $j$,
which we refer to as $j$-cells. The dimension of the complex $X$,
denoted by $d$, is defined as $d:=\max_{\tau\in X}\dim(\tau)$. For
$\ell<d$ , the $\ell$-skeleton of $X$ is the simplicial complex
that consists of all cells of dimension $\leq\ell$ in $X$. The complex
$X$ is said to have a full $\ell$-dimensional skeleton if its $\ell$-skeleton
contains all subsets of $X^{0}\subset V$ of size $\leq$ $\ell+1$.
Throughout the paper we assume that $X$ has a full $\left(d-1\right)$-skeleton
and that $X^{0}=V$.

For $j\geq1$, every $j$-cell $\sigma=\left\{ \sigma^{0},\ldots,\sigma^{j}\right\} $
has two possible orientations, corresponding to the possible orderings
of its vertices, up to an even permutation. Denote an oriented cell
by square brackets, and a flip of orientation by an overline. For
example, one orientation of $\sigma=\left\{ x,y,z\right\} $ is $\left[x,y,z\right]=\left[y,z,x\right]=\left[z,x,y\right]$.
The other orientation is $\overline{\left[x,y,z\right]}=\left[y,x,z\right]=\left[x,z,y\right]=\left[z,y,x\right]$.
Denote by $X_{\pm}^{j}$ the set of oriented $j$-cells (observe that
$|X_{\pm}^{j}|=2\left|X^{j}\right|$ for $j\geq1$) and set $X_{\pm}^{0}=X^{0}$.
Given two oriented cells $\sigma,\sigma'\in X_{\pm}^{d-1},$ let $\sigma\cup\sigma'$
and $\sigma\cap\sigma'$ denote the union and intersection of the
corresponding unoriented cells.

Define the boundary $\partial\sigma$ of the $\left(j+1\right)$-cell
$\sigma=\left\{ \sigma^{0},\ldots,\sigma^{j+1}\right\} \in X^{j+1}$
as the set of $j-$cells obtained by omitting the $i$-th vertex from
$\sigma$, for every $0\leq i\leq j+1$. Namely, 
\[
\partial\sigma=\left\{ \left\{ \sigma^{0},\ldots,\sigma^{i-1},\sigma^{i+1},\ldots,\sigma^{j+1}\right\} \,:\,0\leq i\leq j+1\right\} \subseteq X^{j}.
\]
An oriented $\left(j+1\right)$-cell $\left[\sigma^{0},\ldots,\sigma^{j+1}\right]\in X_{\pm}^{j+1}$
induces orientations on the $j$-cells in its boundary, as follows:
the cell $\left\{ \sigma^{0},\ldots,\sigma^{i-1},\sigma^{i+1},\ldots,\sigma^{j+1}\right\} $
is oriented as $\left(-1\right)^{i}\left[\sigma^{0},\ldots,\sigma^{i-1},\sigma^{i+1},\ldots,\sigma^{j+1}\right]$,
where $-\sigma:=\overline{\sigma}$. As introduced in \cite{PR17}
one can define a neighboring relation on $X_{\pm}^{j}$, where $\sigma,\sigma'\in X_{\pm}^{j}$
are called neighbors, denoted $\sigma\sim\sigma'$, if there exists
an oriented $\left(j+1\right)$-cell, $\tau\in X_{\pm}^{j+1}$, such
that both $\sigma$ and $\overline{\sigma'}$ are in the boundary
of $\tau$ as oriented cells (see Figure \ref{fig:Figure1} for illustration
in the case $d=2$). Observe that this definition guarantees that
for each pair $\left(\sigma,\sigma'\right)\in X^{j}\times X^{j}$
satisfying $\sigma\cup\sigma'\in X^{j+1}$, either $\sigma\sim\sigma'$
or $\sigma\sim\overline{\sigma'}$, but not both.

\begin{center}
	\begin{figure}[h]
	\begin{centering}
		\includegraphics[scale=0.45]{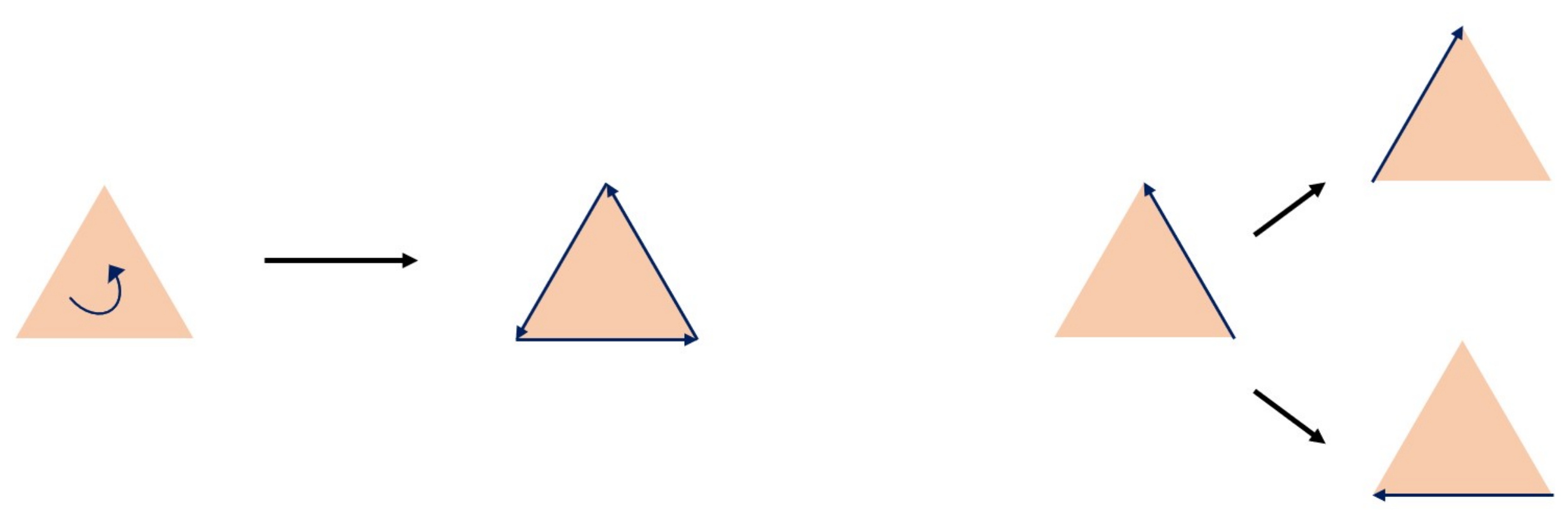}
		\par\end{centering}
		\caption{On the left: an oriented $2$-cell and the orientation it induces
		on its boundary. On the right: an oriented $1$-cell in a $2$-cell
		together with its two oriented neighboring $1$-cells.\label{fig:Figure1}}
	\end{figure}
\par\end{center}

Let $K^{d}:=K\left(d,n\right)$ be the complete $d$-complex with
vertex set $V=[n]:=\left\{ 1,2,\ldots,n\right\} $. That is to say
that $K^{d}$ consists of all subsets of $[n]$ of size $\leq d+1$. 

The Linial-Meshulam model $X=X\left(d,n,p\right)$, with $n,d\in\mathbb{N}$
satisfying $n\geq d+1$ and $p=p\left(n\right)\in\left[0,1\right]$,
is a random $d$-dimensional simplicial complex on $n$ vertices,
with a complete $\left(d-1\right)$-skeleton in which each $d$-cell
of $K^{d}$ is added to $X$ independently with probability $p$.

We fix an arbitrary choice of orientation of the $\left(d-1\right)$-cells
in the complete complex, and denote it as $K_{+}^{d-1}\subset K_{\pm}^{d}$.
Observe that this choice of orientation determine the orientation
of the elements in $X^{d-1}$ since $X^{d-1}=K^{d-1}$. Note that
there is a natural bijection between $K_{+}^{d-1}$ and $K^{d-1}$,
and hence also between $X_{+}^{d-1}$ and $X^{d-1}$. 

The adjacency matrix $A$, associated to the random complex $X$,
is a $|X_{+}^{d-1}|\times|X_{+}^{d-1}|$ random matrix, defined via
\[
A_{\sigma\sigma'}=\begin{cases}
1 & \text{if \ensuremath{\sigma\overset{X}{\sim}\sigma'}}\\
-1 & \text{if \ensuremath{\sigma\overset{X}{\sim}\overline{\sigma'}}}\\
0 & \text{otherwise}
\end{cases},\qquad\forall\sigma,\sigma'\in X_{+}^{d-1}.
\]
Our main result concerns a natural generalization of the centered
and normalized adjacency matrix of $A$, defined by
\[
H_{\sigma\sigma'}:=\begin{cases}
\frac{Z_{\tau}-\mathbb{E}\left[Z\right]}{\sqrt{n\text{Var}\left(Z\right)}} & \text{if }\sigma\overset{K^{d}}{\sim}\sigma'\text{ and \ensuremath{\sigma\cup\sigma'=\tau}}\\
-\frac{Z_{\tau}-\mathbb{E}\left[Z\right]}{\sqrt{n\text{Var}\left(Z\right)}} & \text{if }\sigma\overset{K^{d}}{\sim}\overline{\sigma'}\text{ and \ensuremath{\sigma\cup\sigma'=\tau}}\\
0 & \text{otherwise}
\end{cases},\qquad\forall\sigma,\sigma'\in X_{+}^{d-1},
\]
where $Z$ is a bounded random variable with positive variance and
$\left(Z_{\tau}\right)_{\tau\in K^{d}}$ are i.i.d. copies of it. 
\begin{rem*}
The results stated below for the matrix $H$ hold also for the unsigned
version of the matrix 
\begin{equation}
\tilde{H}_{\sigma\sigma'}=\begin{cases}
\frac{Z_{\sigma\cup\sigma'}-\mathbb{E}\left[Z\right]}{\sqrt{n\text{Var}\left(Z\right)}} & \text{if }\ensuremath{\sigma\cup\sigma'\in X^{d}}\\
0 & \text{otherwise}
\end{cases},\qquad\forall\sigma,\sigma'\in X_{+}^{d-1}.\label{eq:The_matrix_tilde_H}
\end{equation}

Observe that by choosing $Z$ to be a Bernoulli random variable with
parameter $p$, the matrix $H$ reduces to the centered and normalized
adjacency matrix of $X$, defined by 
\[
\mathcal{A}:=\frac{1}{\sqrt{nq}}\left(A-\mathbb{E}\left[A\right]\right),
\]
where 
\[
q:=q(n)=p(1-p).
\]
\end{rem*}
Recall that for a square matrix $M\in\mathbb{R}^{n\times n}$ and
$p\geq1$, the $p-$Schatten norm of $M$ is defined via
\[
\left\Vert M\right\Vert _{S_{p}}=\sqrt[p]{\text{Trace}\left(\left|M\right|^{p}\right)},
\]
where $\left|M\right|=\sqrt{MM^{*}}$. 

We now state our main result.
\begin{thm}
\label{thm:main_result}For every $d\in\mathbb{N}$, there exist constants
$C_{d},c_{d}\in(0,\infty)$ depending only on $d$, such that for
every $n\geq d+1$ and any integer $k:=k\left(n\right)\geq d$
\[
\sqrt[2k]{\mathbb{E}\left[\left\Vert H\right\Vert _{S_{2k}}^{2k}\right]}\leq\Phi\left(\theta_{k},\theta_{k}^{*}\right),
\]
where 
\[
\theta_{k}:=\sqrt{\frac{n-d}{n}}{n \choose d}^{\frac{1}{2k}}\qquad,\qquad\theta_{k}^{*}:=\left\Vert Z_{\tau}-\mathbb{E}\left[Z\right]\right\Vert _{\infty}\left({n \choose d}\cdot\frac{d\left(n-d\right)}{\left(n\text{Var}\left(Z\right)\right)^{k}}\right)^{\frac{1}{2k}},
\]
and
\[
\Phi\left(x,y\right):=\sqrt[2k]{d!d}\cdot y\left(\frac{x}{y}+2\sqrt{k}\right)^{\frac{d-1}{k}}\left(2\sqrt{d}\left(\frac{x}{y}+2\sqrt{k}\right)+C_{d}\left(\frac{x}{y}+2\sqrt{k}\right)^{\frac{2}{3}}\Big(\log\left(\frac{x}{y}+2\sqrt{k}\right)\Big)^{2/3}+c_{d}\sqrt{k}\right).
\]
\end{thm}

Theorem \ref{thm:main_result} allows us to control the operator norm
of $H$.
\begin{cor}
\label{cor:Norm_of_H} For every $d\in\mathbb{N}$, if $n\text{Var}\left(Z\right)\gg\log\left(n\right)$,
namely $\lim_{n\to\infty}\frac{\log(n)}{n\text{Var}(Z)}=0$, then
\[
\lim_{n\to\infty}\mathbb{E}\left[\left\Vert H\right\Vert _{2}\right]=2\sqrt{d}.
\]
\end{cor}

Furthermore, it provides an upper bound on the probability that the
operator norm of $H$ is bigger than $2\sqrt{d}$.
\begin{cor}
\label{cor:Norm_of_H_2} For every $d\in\mathbb{N}$, there exists
a constant $C_{d}\in(0,\infty)$, depending only on $d$, such that
if $n\text{Var}\left(Z\right)\gg\log\left(n\right)$, then for all
$\varepsilon>0$ and all large enough $n$ (depending only on $\varepsilon$
and $d$)
\[
\mathbb{P}\left(\left\Vert H\right\Vert _{2}\geq2\sqrt{d}+\varepsilon\right)\leq e^{-C_{d}n\text{Var}\left(Z\right)\varepsilon^{2}},
\]
and hence 
\[
\lim_{n\to\infty}\left\Vert H\right\Vert _{2}=2\sqrt{d},\qquad\mathbb{P}\text{-a.s.}
\]
\end{cor}

Finally, following the argument in \cite[Theorem 2.1 and Corollary 2.3]{RK17}
and using Corollary \ref{cor:Norm_of_H_2} allow us to extend the
spectral gap result obtained in \cite{RK17} for the matrix $A$ in
the regime $nq\gg\log^{4}(n)$, to the regime $nq\gg\log(n)$.
\begin{thm}[Eigenvalue confinement]
\label{thm:eigenvalue_confinement} For every $d\geq2$, there exists
a positive constant $C>0$ depending only on $d$, such that the following
holds with probability at least $1-n^{-D}$ for all $D>0$, provided
$nq\gg\log\left(n\right)$.
\begin{enumerate}
\item For every $\xi>0$, and all large enough $n$ (depending on $D$,
$\xi$ and $d$), the ${n-1 \choose d}$ smallest eigenvalues of the
matrix $A$ are within the interval $\sqrt{dnq}\left[-2-\xi,2+\xi\right]$
.
\item If $q\log^{6}\left(n\right)\leq\frac{1}{C\left(1+D\right)^{6}}$,
then for all large enough $n$ (depending on $D$ and $d$), the remaining
${n-1 \choose d-1}$ eigenvalues of $A$ lie in the interval $nq+\left[-7d,7d\right]$.
\end{enumerate}
\end{thm}

As an immediate corollary from the last theorem we obtain 
\begin{cor}[Spectral gap]
\label{cor:Spectral_gap} For every $d\geq2$, there exists a positive
constant $C>0$ depending only on $d$ such that for all $\xi>0$,
$D>0$ satisfying $nq\gg\log\left(n\right)$ and $q\log^{6}\left(n\right)\leq\frac{1}{C\left(1+D\right)^{6}}$,
we have for all $n$ large enough (depending on $d$, $D$ and $\xi$)
\[
\lambda_{{n-1 \choose d}+1}-\lambda_{{n-1 \choose d}}=nq-2\sqrt{dnq}\left(1+O\left(\xi\right)\right),
\]
with probability at least $1-n^{-D}$.
\end{cor}

\paragraph{Conventions. \textmd{Throughout the rest of the paper we use $C$
to denote a generic large positive constant, which may depend on some
fixed parameters and whose value may change from one expression to
the next. If $C$ depends on some parameter $k$, we sometimes emphasize
this dependence by writing $C_{k}$ instead of $C$. The letters $d,i,j,k,l,m,n,r,s,N$
are always used to denote an element in $\mathbb{N}:=\{0,1,2,\ldots\}$.
From now on, we consistently use $\sigma$ for (oriented or non-oriented)
$(d-1)$-cells, and $\tau$ for (oriented or non-oriented) $d$-cells.}}

\section{\label{sec:Norm-bounds-for}Norm bounds for the unsigned adjacency
matrix }

In this section we introduce two useful bounds, having a significant
role in the proof of Theorem \ref{thm:main_result}. 

Let $r\geq d+1$ and $p_{0}\in(0,1)$ a fixed number which does not
depend on $r$. Denote by $Y_{{r \choose d}\times{r \choose d}}$
the matrix obtained from \eqref{eq:The_matrix_tilde_H} by taking $\text{Ber}(p_{0})$
distribution, that is, $Y$ is the unoriented normalized adjacency
matrix arising from $X(d,r,p_{0})$, given by 
\[
Y_{\sigma\sigma'}=\begin{cases}
\frac{\chi_{\sigma\cup\sigma'}-p_{0}}{\sqrt{q_{0}}} & \text{if \ensuremath{\sigma\cup\sigma'\in X^{d}}}\\
0 & \text{otherwise}
\end{cases},\qquad\forall\sigma,\sigma'\in X_{+}^{d-1},
\]
where $q_{0}:=p_{0}\left(1-p_{0}\right)$ and $\left(\chi_{\tau}\right)_{\tau\in X^{d}}$
are i.i.d. Ber$(p_{0})$ random variables. With a slight abuse of
notation, we use $\left\Vert \cdot\right\Vert _{2}$ to denote both
the Euclidean norm when applied to a vector $\mathbf{v}\in\mathbb{R}^{m}$,
and the operator norm when applied to a real matrix.

\subsection{Bounding the expected value of the norm of $Y$}
\begin{prop}
\label{prop:first_Bound_on_Y}For every $p_{0}\in(0,1)$ there exists
a constant $C=C_{d,q_{0}}\in(0,\infty)$ such that for all $r\geq d+1$
\[
\mathbb{E}\left[\left\Vert Y\right\Vert _{2}\right]\leq2\sqrt{dr}+C_{d,q_{0}}r^{1/3}\log^{2/3}r.
\]
\end{prop}

\begin{proof}
We first prove the inequality holds for all sufficiently large $r$
(depending only on $d$ and $q_{0}$). Using the CDF formula for calculating
expectation gives for every $\alpha>0$

\begin{alignat}{1}
\mathbb{E}\left[\left\Vert Y\right\Vert _{2}\right] & =\int_{0}^{\infty}\mathbb{P}\left(\left\Vert Y\right\Vert _{2}>t\right)dt\nonumber \\
 & \leq2\sqrt{dr}+\alpha r^{1/3}\log^{2/3}r+\sqrt{r}\int_{\alpha r^{-1/6}\log^{2/3}r}^{\infty}\mathbb{P}\left(\left\Vert Y\right\Vert _{2}>2\sqrt{dr}+u\sqrt{r}\right)du.\label{eq:33}
\end{alignat}
By \cite[Theorem 5.1]{RK17}\footnote{Note that in \cite{RK17} the authors consider the oriented adjacency
matrix of the complex $X\left(d,n,p\right)$ (denoted by $A)$, which
takes into account the orientation of each $\left(d-1\right)-$cell.
However, going over the proof of Theorem $5.1$, one can verify that
it remains valid for the matrix $Y$.}, assuming $rq_{0}\geq2$, for every $u>0$
\begin{align*}
\mathbb{P}\left(\left\Vert Y\right\Vert _{2}>2\sqrt{dr}+u\sqrt{r}\right) & \le\frac{2}{(d-1)!}\xi(d,r,u)\\
 & :=\frac{2\left(1+\frac{u}{2\sqrt{d}}\right)^{2}}{\left(d-1\right)!}\exp\left(d\log\left(r\right)-\left(\frac{2}{3}\log\left(1+\frac{u}{2\sqrt{d}}\right)\right)^{3/2}\left(\frac{rq_{0}}{d}\right)^{1/4}\right),
\end{align*}
and thus 
\begin{alignat}{1}
\mathbb{E}\left[\left\Vert Y\right\Vert _{2}\right]\leq2\sqrt{dr}+\alpha r^{1/3}\log^{2/3}r+ & \frac{2\sqrt{r}}{\left(d-1\right)!}\int_{\alpha r^{-1/6}\log^{2/3}r}^{\infty}\xi(d,r,u)du.\label{eq:26}
\end{alignat}
Given $\varepsilon>0$, by choosing $\alpha>2d^{4/3}q_{0}^{-1/6}\varepsilon^{-2/3}$,
one can verify that for all $u\geq\alpha r^{-1/6}\log^{2/3}r$ and
all sufficiently large $r$ (depending only on $d$ and $q_{0}$ and
the choice of $\alpha$)
\[
\frac{d\log(r)}{\log^{3/2}\left(1+\frac{u}{2\sqrt{d}}\right)\left(\frac{rq_{0}}{d}\right)^{1/4}}\leq\varepsilon.
\]
In particular, taking $\varepsilon=\left(\frac{2}{3}\right)^{3/2}-\left(\frac{1}{3}\right)^{3/2}$
and choosing an appropriate $\alpha$, for all sufficiently large
$r$ (depending only on $d$ and $q_{0})$ 
\[
\xi(d,r,u)\leq\left(1+\frac{u}{2\sqrt{d}}\right)^{2}\exp\left(-\left(\frac{1}{3}\left(\frac{rq_{0}}{d}\right)^{1/6}\log\left(1+\frac{u}{2\sqrt{d}}\right)\right)^{3/2}\right).
\]
Consequently, for all $r$ satisfying $\frac{1}{3^{3/2}}\left(\frac{rq_{0}}{d}\right)^{1/4}\geq1$
\begin{align*}
\int_{\alpha r^{-1/6}\log^{2/3}r}^{\infty}\xi(d,r,u)du\leq & \int_{1+\frac{1}{2}\alpha d^{-1/2}r^{-1/6}\log^{2/3}r}^{\infty}s^{2}\exp\left(-\left(\frac{1}{3}\left(\frac{rq_{0}}{d}\right)^{1/6}\log s\right)^{3/2}\right)ds\\
\leq & \int_{1+\frac{1}{2}\alpha d^{-1/2}r^{-1/6}\log^{2/3}r}^{\infty}s^{2}\exp\left(-\frac{1}{3}\left(\frac{rq_{0}}{d}\right)^{1/6}\log s\right)ds\\
\leq & \int_{1}^{\infty}s^{2}\exp\left(-\frac{1}{3}\left(\frac{rq_{0}}{d}\right)^{1/6}\log s\right)ds\\
= & \int_{1}^{\infty}s^{2-\frac{1}{3}\left(\frac{rq_{0}}{d}\right)^{1/6}}ds=\frac{1}{\frac{1}{3}\left(\frac{rq_{0}}{d}\right)^{1/6}-3},
\end{align*}
where in the second inequality we used the fact that $\frac{1}{3}\left(\frac{rq_{0}}{d}\right)^{1/6}\log\left(1+\frac{u}{2\sqrt{d}}\right)\geq1$
for all $u\geq1+\frac{1}{2}\alpha d^{-1/2}r^{-1/6}\log^{2/3}r$ provided
$r$ is sufficiently large (depending only on $d$ and $q_{0}$).
Combining the last estimation with \eqref{eq:26} and \eqref{eq:33},
gives 
\[
\mathbb{E}\left[\left\Vert Y\right\Vert _{2}\right]\leq2\sqrt{dr}+\alpha r^{1/3}\log^{2/3}r+\frac{2\sqrt{r}}{\left(d-1\right)!}\frac{1}{\frac{1}{3}\left(\frac{rq_{0}}{d}\right)^{1/6}-3}\leq2\sqrt{dr}+C_{d,q_{0}}r^{1/3}\log^{2/3}r.
\]
In order to obtain the result for all $r\geq d+1$, we note that by
increasing the value of $C_{d,q_{0}}$ even further it follows that
the last inequality holds for any $r\geq d+1$, thus concluding
the proof of Proposition \ref{prop:first_Bound_on_Y}.
\end{proof}

\subsection{Bounding the expected value of powers of the norm of $Y$}
\begin{prop}
\label{propr:Second_estimation_on_Y} For every $p_{0}\in(0,1)$,
there exists a constant $C_{d,q_{0}}\in(0,\infty)$ (depending only
on $d$ and $q_{0}$), such that for all $k\in\mathbb{N}$ and all
$r\geq d+1$

\textbf{
\[
\left(\mathbb{E}\left[\left\Vert Y\right\Vert _{2}^{2k}\right]\right)^{\frac{1}{2k}}\leq\mathbb{E}\left[\left\Vert Y\right\Vert _{2}\right]+C_{d,q_{0}}\sqrt{k}.
\]
}
\end{prop}

The proof of Proposition \ref{propr:Second_estimation_on_Y} is based
on the following concentration result.
\begin{lem}
\label{lem:technical_lemma_on_Y} For every $p_{0}\in(0,1)$, there
exists a constant $c_{d,q_{0}}\in(0,\infty)$, depending only on $d$
and $q_{0}$, such that for any $t>0$ and $r\geq d+1$
\[
\mathbb{P}\left(\left\Vert Y\right\Vert _{2}\geq\mathbb{E}\left[\left\Vert Y\right\Vert _{2}\right]+t\right)\leq e^{-c_{d,q_{0}}t^{2}}.
\]
\end{lem}

\begin{proof}[Proof of Lemma \ref{lem:technical_lemma_on_Y}]
 The inequality follows from Talagrand\textquoteright s concentration
inequality \cite[Theorem 6.10]{BLM13} using the fact that the function
$f_{p_{0},r}:[0,1]^{|K^{d}|}\to\mathbb{R}$, defined by 
\[
f_{p_{0},r}((x_{\tau})_{\tau\in K^{d}})=\sqrt{\frac{q_{0}}{d\left(d+1\right)}}\|A((x_{\tau})_{\tau\in K^{d}})\|_{2},
\]
where $A((x_{\tau})_{\tau\in K^{d}})$ is a $|X_{+}^{d-1}|\times|X_{+}^{d-1}$|
matrix defined by
\[
A_{\sigma\sigma'}\left((x_{\tau})_{\tau\in K^{d}}\right)=\begin{cases}
0 & \text{if \ensuremath{\sigma\cup\sigma'\notin K^{d}} }\\
\frac{x_{\tau}-p_{0}}{\sqrt{p_{0}}} & \text{if \ensuremath{\sigma\cup\sigma'=\tau\in K^{d}} }
\end{cases},\qquad\forall\sigma,\sigma'\in X_{+}^{d-1}
\]
is a convex 1-Lipschitz function, and therefore, for any $t>0$
\begin{alignat}{1}
 & \mathbb{P}\left(\left\Vert Y\right\Vert _{2}\geq\mathbb{E}\left[\left\Vert Y\right\Vert _{2}\right]+t\right)\nonumber \\
= & \mathbb{P}\left(f_{p_{0},r}\left(\left(\chi_{\tau}\right)_{\tau\in K^{d}}\right)\geq\mathbb{E}\left[f_{p_{0},r}\left(\left(\chi_{\tau}\right)_{\tau\in K^{d}}\right)\right]+\sqrt{\frac{q_{0}}{d\left(d+1\right)}}t\right)\leq e^{-c_{d,q_{0}}t^{2}},\label{eq:34}
\end{alignat}
where $c_{d,q_{0}}:=\frac{q_{0}}{2d\left(d+1\right)}$.
\end{proof}
\begin{proof}[Proof of Proposition \ref{propr:Second_estimation_on_Y}]
 We strive towards estimating $\mathbb{E}\Big[\left\Vert Y\right\Vert _{2}^{2k}\Big]$.
Using the CDF formula for expectation gives
\begin{alignat}{1}
\mathbb{E}[\left\Vert Y\right\Vert _{2}^{2k}] & =\int_{0}^{\infty}\mathbb{P}\left(\left\Vert Y\right\Vert _{2}^{2k}>t\right)dt\nonumber \\
 & =\int_{0}^{\left(\mathbb{E}\left[\left\Vert Y\right\Vert _{2}\right]\right)^{2k}}\mathbb{P}\left(\left\Vert Y\right\Vert _{2}^{2k}>t\right)dt+\int_{\left(\mathbb{E}\left[\left\Vert Y\right\Vert _{2}\right]\right)^{2k}}^{\infty}\mathbb{P}\left(\left\Vert Y\right\Vert _{2}^{2k}>t\right)dt\nonumber \\
 & \leq\left(\mathbb{E}\left[\left\Vert Y\right\Vert _{2}\right]\right)^{2k}+2k\int_{0}^{\infty}\left(\mathbb{E}\left[\left\Vert Y\right\Vert _{2}\right]+\eta\right)^{2k-1}e^{-c_{d,q_{0}}\eta^{2}}d\eta,\label{eq:19}
\end{alignat}
where in the second integral we used the change of variable $t=\left(\mathbb{E}\left[\left\Vert Y\right\Vert _{2}\right]+\eta\right)^{2k}$
together with Lemma \ref{lem:technical_lemma_on_Y}. Using the binomial
formula, the integral in \eqref{eq:19} gives
\begin{alignat}{1}
\mathbb{E}[\left\Vert Y\right\Vert _{2}^{2k}] & \leq\left(\mathbb{E}\left[\left\Vert Y\right\Vert _{2}\right]\right)^{2k}+2k\sum_{j=0}^{2k-1}\left(\mathbb{E}\left[\left\Vert Y\right\Vert _{2}\right]\right)^{j}{2k-1 \choose j}\int_{0}^{\infty}\eta^{2k-1-j}e^{-c_{d,q_{0}}\eta^{2}}d\eta\nonumber \\
 & \overset{\left(1\right)}{=}\left(\mathbb{E}\left[\left\Vert Y\right\Vert _{2}\right]\right)^{2k}+\sum_{j=0}^{2k-1}\left(2k-j\right){2k \choose j}\left(\mathbb{E}\left[\left\Vert Y\right\Vert _{2}\right]\right)^{j}\frac{1}{\left(2c_{d,q_{0}}\right)^{\frac{2k-j}{2}}}\int_{0}^{\infty}t^{2k-1-j}e^{-\frac{t^{2}}{2}}dt\nonumber \\
 & \overset{\left(2\right)}{=}\left(\mathbb{E}\left[\left\Vert Y\right\Vert _{2}\right]\right)^{2k}+\sum_{j=0}^{2k-1}\left(2k-j\right){2k \choose j}\left(\mathbb{E}\left[\left\Vert Y\right\Vert _{2}\right]\right)^{j}\frac{1}{\left(2c_{d,q_{0}}\right)^{\frac{2k-j}{2}}}\sqrt{\frac{\pi}{2}}\int_{-\infty}^{\infty}\frac{1}{\sqrt{2\pi}}\left|t\right|^{2k-1-j}e^{-\frac{t^{2}}{2}}dt\nonumber \\
 & \overset{\left(3\right)}{\leq}\left(\mathbb{E}\left[\left\Vert Y\right\Vert _{2}\right]\right)^{2k}+\sum_{j=0}^{2k-1}\left(2k-j\right){2k \choose j}\left(\mathbb{E}\left[\left\Vert Y\right\Vert _{2}\right]\right)^{j}\frac{1}{\left(2c_{d,q_{0}}\right)^{\frac{2k-j}{2}}}\sqrt{\frac{\pi}{2}}\left(2k-1-j\right)!!\nonumber \\
 & \leq\left(\mathbb{E}\left[\left\Vert Y\right\Vert _{2}\right]\right)^{2k}+\sqrt{\frac{\pi}{c_{d,q_{0}}}}k\left(\mathbb{E}\left[\left\Vert Y\right\Vert _{2}\right]\right)^{2k-1}+\sum_{j=0}^{2k-2}\left(2k-j\right){2k \choose j}\left(\mathbb{E}\left[\left\Vert Y\right\Vert _{2}\right]\right)^{j}\frac{2}{\left(2c_{d,q_{0}}\right)^{\frac{2k-j}{2}}}e^{\left(\frac{2k-j}{2}\right)\log\left(\left(2k-1-j\right)\right)},\label{eq:20}
\end{alignat}
where in $\left(1\right)$ we used the change of variables $\eta=(2c_{d,q_{0}})^{-1/2}t$,
$\left(2\right)$ holds since the integrand is an even function and
$\left(3\right)$ is due to the central absolute moment formula of
a standard normal random variable. 

We wish to show that the expression in \eqref{eq:20} is bounded from
above by $\left(\mathbb{E}\left[\left\Vert Y\right\Vert _{2}\right]+C_{d,q_{0}}\sqrt{k}\right)^{2k}$,
for an appropriate choice of positive constant $C_{d,q_{0}}$, which
depends only on $d$ and $q_{0}$. We will show this for $C_{d,q_{0}}=\sqrt{\frac{e^{4}}{c_{d,q_{0}}}}$
by showing that each term in \eqref{eq:20} is bounded from above
by the corresponding term in $\sum_{j=0}^{2k}{2k \choose j}\left(\mathbb{E}\left[\left\Vert Y\right\Vert _{2}\right]\right)^{j}\left(C_{d,q_{0}}\sqrt{k}\right)^{2k-j}$
(which by the binomial formula equals $\left(\mathbb{E}\left[\left\Vert Y\right\Vert _{2}\right]+C_{d,q_{0}}\sqrt{k}\right)^{2k}$).
For $j=2k$ both summands are equal. As for $j=2k-1$ note that in
\eqref{eq:20} we obtain $\sqrt{\frac{\pi}{c_{d,q_{0}}}}k\left(\mathbb{E}\left[\left\Vert Y\right\Vert _{2}\right]\right)^{2k-1}$,
while in the Binomial formula we obtain $2k\left(\mathbb{E}\left[\left\Vert Y\right\Vert _{2}\right]\right)^{2k-1}C_{d,q_{0}}\sqrt{k}$.
Thus, the result trivially holds by taking $C_{d,q_{0}}\geq\frac{1}{2}\sqrt{\frac{\pi}{c_{d,q_{0}}}}.$
Finally for $0\leq j\leq2k-2$, we observe the following equivalent
statements:
\begin{alignat}{1}
 & \left(2k-j\right){2k \choose j}\left(\mathbb{E}\left[\left\Vert Y\right\Vert _{2}\right]\right)^{j}\frac{2}{\left(2c_{d,q_{0}}\right)^{\frac{2k-j}{2}}}e^{\left(\frac{2k-j}{2}\right)\log\left(2k-1-j\right)}\leq{2k \choose j}\left(\mathbb{E}\left[\left\Vert Y\right\Vert _{2}\right]\right)^{j}\left(C_{d,q_{0}}\sqrt{k}\right)^{2k-j}\nonumber \\
\Leftrightarrow & \left(2k-j\right)\frac{2}{\left(2c_{d,q_{0}}\right)^{\frac{2k-j}{2}}}e^{\left(\frac{2k-j}{2}\right)\log\left(2k-1-j\right)}\leq\left(C_{d,q_{0}}\sqrt{k}\right)^{2k-j}\nonumber \\
\Leftrightarrow & \frac{2}{2k-j}\log\left(4k-2j\right)+\log\left(\frac{2k-1-j}{2c_{d,q_{0}}}\right)\leq\log\left(C_{d,q_{0}}^{2}k\right)\label{eq:21-1}
\end{alignat}
The LHS of \eqref{eq:21-1} is bounded from above by
\[
  \frac{2}{2k-j}\log\left(4k-2j\right)+\log\left(\frac{k}{c_{d,q_{0}}}\right)
\leq  \log\left(\frac{k}{c_{d,q_{0}}}\right)+4
=  \log\left(\frac{e^{4}k}{c_{d,q_{0}}}\right),
\]

which is the expression on the right hand side of \eqref{eq:21-1}.
Together with \eqref{eq:19} and \eqref{eq:20} we obtain 
\begin{align*}
\mathbb{E}[\left\Vert Y\right\Vert _{2}^{2k}] & \leq\sum_{j=0}^{2k}{2k \choose j}\left(\mathbb{E}\left[\left\Vert Y\right\Vert _{2}\right]\right)^{j}\left(C_{d,q_{0}}\sqrt{k}\right)^{2k-j}\\
 & =\left(\mathbb{E}\left[\left\Vert Y\right\Vert _{2}\right]+C_{d,q_{0}}\sqrt{k}\right)^{2k},
\end{align*}
which concludes the proof of Proposition \ref{propr:Second_estimation_on_Y}.
\end{proof}

\section{Bounding the $2k$-Schatten norm of $H$}

\subsection{Proof of Theorem \ref{thm:main_result}}

Let us start with several definitions that are used throughout the
proof. See Example \ref{exa:example_of_definitions} for an illustration. 
\begin{defn}
(Word) A \emph{letter} is an element of $X_{+}^{d-1}$. A \emph{word
}of length $m\in\mathbb{N}$, is a finite sequence $\sigma_{1}\sigma_{2}\ldots\sigma_{m}$
of letters, at least one letter long, such that $\sigma_{i}\cup\sigma_{i+1}\in X^{d}$
for all $1\leq i\leq m-1$. A word is called closed if its first and
last letters are the same, namely $\sigma_{1}=\sigma_{m}$. Two words
of the same length $w=\sigma_{1}\ldots\sigma_{m}$ and $w'=\sigma'_{1}\ldots\sigma'_{m}$
are called \emph{equivalent,} denoted as $w\sim w'$, if there exists
a permutation $\pi$ on $V=X^{0}=\left[n\right]$ such that $\pi\left(\sigma_{i}\right)=\sigma'_{i}$
for every $1\leq i\leq m$, where for $\sigma=\left[\sigma^{0},\sigma^{1},\cdots,\sigma^{d-1}\right]\in X_{\pm}^{d-1}$
we write $\pi\left(\sigma_{i}\right)=\left[\pi\left(\sigma^{0}\right),\pi\left(\sigma^{1}\right),\cdots,\pi\left(\sigma^{d-1}\right)\right]$.
\end{defn}

\begin{defn}
(Support) For a word $w=\sigma_{1}\ldots\sigma_{m}$, we define its
\emph{support} by $\text{supp}_{0}\left(w\right)=\bigcup_{i=1}^{m}\sigma_{i}\subseteq V$,
and its \emph{$d-$cell support} by $\text{supp}_{d}\left(w\right)=\left\{ \sigma_{i}\cup\sigma_{i+1};1\leq i\leq m-1\right\} \subseteq K^{d}.$
\end{defn}

\begin{defn}
(Graph of a word). Given a word $w=\sigma_{1}\ldots\sigma_{m}$, define
$G_{w}=\left(V_{w},E_{w}\right)$ to be the graph with vertex set
$V_{w}=\left\{ \sigma_{i};1\leq i\leq m\right\} \subseteq X_{+}^{d-1}$
and edge set $E_{w}=\left\{ \left\{ \sigma_{i},\sigma_{i+1}\right\} ;1\leq i\leq m-1\right\} \subseteq K^{d}$.
Let $\mathbb{G}$ denote the collection of all labeled, undirected
graphs induced from words. Namely, $\mathbb{G}:=\left\{ G_{w}\text{ : \ensuremath{w} is a word}\right\} .$
The graph $G_{w}$ comes with a path, given by the word $w$, that
goes through all of its vertices and edges. We call each step along
the path, i.e., $\sigma_{i}\sigma_{i+1}$ for some $1\leq i\leq m-1$,
a\emph{ crossing }of the edge $\left\{ \sigma_{i},\sigma_{i+1}\right\} $
and a \emph{crossing} of the $d$-cell $\sigma_{i}\cup\sigma_{i+1}$.
For an edge $e\in E_{w}$, define $N_{w}\left(e\right)$ to be the
number of times the edge $e$ is crossed along the path generated
by $w$ in the graph $G_{w}$. For a $d-$cell $\tau\in\text{supp}_{d}\left(w\right)$,
let $\mathcal{E}_{w}\left(\tau\right):=\left\{ \left\{ \sigma,\sigma'\right\} \in E_{w};\sigma\cup\sigma'=\tau\right\} $
and define

\[
N_{w}\left(\tau\right)=\sum_{e\in\mathcal{E}_{w}\left(\tau\right)}N_{w}\left(e\right)
\]
to be the total number of times the $d-$cell is crossed along the
path generated by the word $w$. 
\end{defn}

\begin{example}
\label{exa:example_of_definitions}In the case $d=2$, for $w_{1}=\left[6,5\right]\left[6,7\right]\left[6,5\right]$
and $w_{2}=\left[2,1\right]\left[3,1\right]\left[2,1\right]$, we
have $w_{1}\sim w_{2}$ via any permutation on $\left[n\right]$ satisfying
$6\leftrightarrow1,5\leftrightarrow2,7\leftrightarrow3$ (see Figure
\ref{fig:Left:-The-path}).
\begin{figure}[h]
\begin{centering}
\includegraphics[scale=0.5]{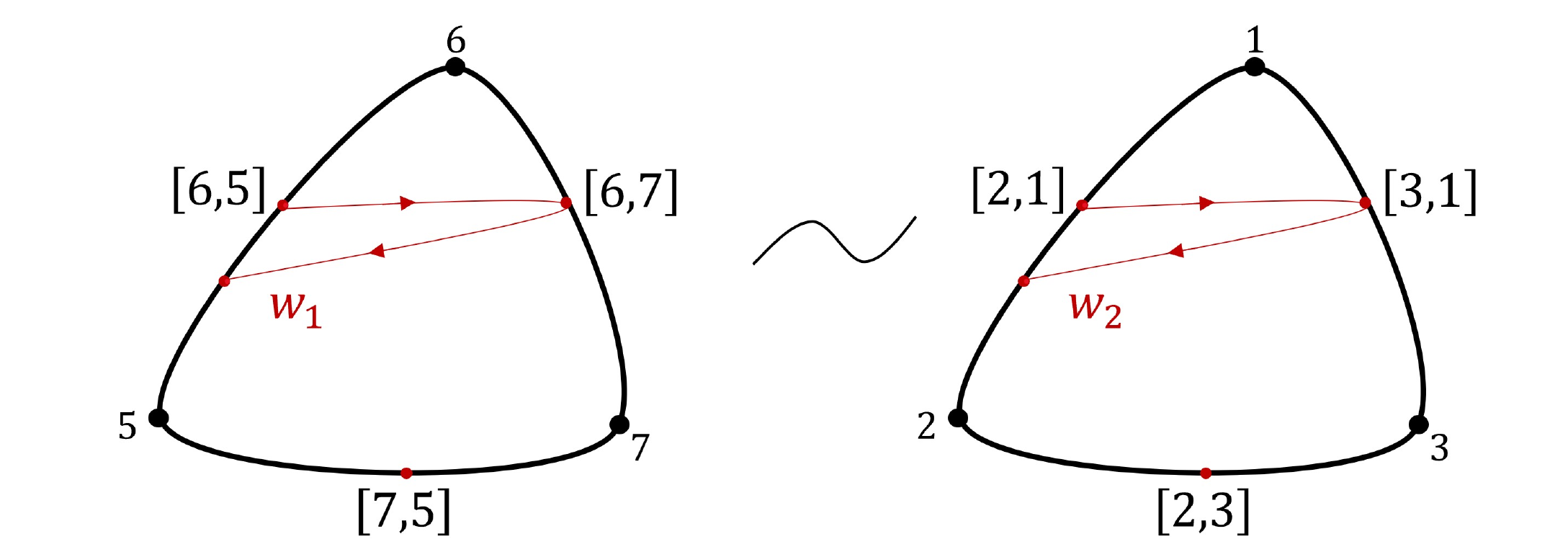}
\par\end{centering}
\caption{Left: The path generated from the word $w_{1}$. Right: The path generated
from the word $w_{2}$ which is equivalent to $w_{1}$.\label{fig:Left:-The-path}}
\end{figure}
Furthermore, the support of the word $w=\left[5,6\right]\left[6,7\right]\left[5,6\right]$$\left[5,8\right]$
is given by $\text{supp}_{0}\left(w\right)=\left\{ 5,6,7,8\right\} $
and its $2$-cell support by $\text{supp}_{2}\left(w\right)=\left\{ \{5,6,7\},\{6,5,8\}\right\} $.
\end{example}

Let $B:=\sqrt{n\text{Var}\left(Z\right)}H.$ Since $H$ is symmetric,
for every $k\in\mathbb{N}$,
\begin{alignat}{1}
\mathbb{E}[\left\Vert H\right\Vert _{S_{2k}}^{2k}] & =\mathbb{E}\left[\text{Tr}\left(H^{2k}\right)\right]\nonumber \\
 & =\mathbb{E}\left[\text{Tr}\left(\left((n\text{Var}\left(Z\right))^{-1/2}B\right)^{2k}\right)\right]\nonumber \\
 & =\frac{1}{\left(n\text{Var}\left(Z\right)\right)^{k}}\mathbb{E}\left[\text{Tr}\left(B^{2k}\right)\right]\nonumber \\
 & =\frac{1}{\left(n\text{Var}\left(Z\right)\right)^{k}}\sum_{\sigma_{1},\ldots,\sigma_{2k}\in X_{+}^{d-1}}\mathbb{E}\left[B_{\sigma_{1}\sigma_{2}}B_{\sigma_{2}\sigma_{3}}\cdots B_{\sigma_{2k}\sigma_{1}}\right].\label{eq:0.1}
\end{alignat}
Each term in the sum, $B_{\sigma_{1}\sigma_{2}}B_{\sigma_{2}\sigma_{3}}\cdots B_{\sigma_{2k}\sigma_{1}}$,
can be associated with a string of letters $\sigma_{1}\sigma_{2}\ldots\sigma_{2k}$.
Since $B_{\sigma\sigma'}=0$ whenever $\sigma\cup\sigma'\notin X^{d}$,
it follows that the list of letters which contribute to the sum in
\eqref{eq:0.1} are the set of closed words of length $2k+1$. Using
the independent structure of $H$ for different $d-$cells and the
definition of $N_{w}\left(\tau\right)$ we then have
\begin{alignat*}{1}
\mathbb{E}[\left\Vert H\right\Vert _{S_{2k}}^{2k}] & =\frac{1}{\left(n\text{Var}\left(Z\right)\right)^{k}}\sum_{\tiny\begin{array}{c}
\text{\ensuremath{w} a closed word}\\
\text{of length \ensuremath{2k+1}}
\end{array}\tiny}\prod_{\tau\in K^{d}}\mathbb{E}\left[B_{\tau}^{N_{w}\left(\tau\right)}\right]\\
 & \leq\frac{1}{\left(n\text{Var}\left(Z\right)\right)^{k}}\sum_{\tiny\begin{array}{c}
\text{\ensuremath{w} a closed word}\\
\text{of length \ensuremath{2k+1}}
\end{array}\tiny}\prod_{\tau\in K^{d}}\left|\mathbb{E}\left[B_{\tau}^{N_{w}\left(\tau\right)}\right]\right|,
\end{alignat*}
where $B_{\tau}:=B_{\sigma\sigma'}$ for some $\sigma,\sigma'\in X_{+}^{d-1}$
with $\sigma\cup\sigma'=\tau$ (observe that the value of $\left|\mathbb{E}\left[B_{\tau}^{N_{w}\left(\tau\right)}\right]\right|$
for any pair $\left(\sigma,\sigma'\right)\in X_{+}^{d-1}$ with $\sigma\cup\sigma'=\tau$,
is the same, hence the last expression is well defined). 

Note that if $N_{w}\left(\tau\right)=1$, then
\[
\mathbb{E}\left[B_{\tau}^{N_{w}\left(\tau\right)}\right]=\mathbb{E}\left[B_{\tau}\right]=\mathbb{E}\left[Z_{\tau}-\mathbb{E}\left[Z\right]\right]=0,
\]
and hence we only need to address closed words of length $2k+1$ such
that 
\begin{equation}
N_{w}\left(\tau\right)\geq2,\qquad\forall\tau\in\text{supp}_{d}\left(w\right).\label{eq:0}
\end{equation}

Denote by $\mathcal{W}_{2k+1}$ a set of representatives for the equivalence
classes of closed words of length $2k+1$ with $N_{w}\left(\tau\right)\geq2$
for all $\tau\in\text{supp}_{d}\left(w\right)$. As a consequence
of the above remark we obtain 
\begin{alignat*}{1}
\mathbb{E}\left[\left\Vert H\right\Vert _{S_{2k}}^{2k}\right] & \leq\frac{1}{\left(n\text{Var}\left(Z\right)\right)^{k}}\sum_{w\in\mathcal{W}_{2k+1}}\sum_{u\sim w}\prod_{\tau\in\text{supp}_{d}\left(u\right)}\mathbb{E}\left[\left|B_{\tau}\right|^{N_{u}\left(\tau\right)}\right].
\end{alignat*}
For $\tau\in K^{d}$ and $m\in\mathbb{N}$, define 
\begin{equation}
b_{\tau}^{\left(m\right)}:=\mathbb{E}\left[\left|B_{\tau}\right|^{m}\right],\label{eq:5.1}
\end{equation}
and note that despite the notation $b_{\tau}^{\left(m\right)}$ is
independent of $\tau$, since all none zero entries of $B$ have the
same distribution. Moreover, $\left|\text{supp}_{d}\left(u\right)\right|=\left|\text{supp}_{d}\left(w\right)\right|$
for any two equivalent words $u$ and $w$. Consequently
\begin{alignat*}{1}
\mathbb{E}\left[\left\Vert H\right\Vert _{S_{2k}}^{2k}\right] & \leq\frac{1}{\left(n\text{Var}\left(Z\right)\right)^{k}}\sum_{w\in\mathcal{W}_{2k+1}}\sum_{u\sim w}\prod_{\tau\in\text{supp}_{d}\left(w\right)}b_{\tau}^{\left(N_{\ensuremath{w}}\left(\tau\right)\right)}\\
 & =\frac{1}{\left(n\text{Var}\left(Z\right)\right)^{k}}\sum_{w\in\mathcal{W}_{2k+1}}\prod_{\tau\in\text{supp}_{d}\left(w\right)}b_{\tau}^{\left(N_{\ensuremath{w}}\left(\tau\right)\right)}\left|\left\{ u:u\text{ is a word such that }u\sim w\right\} \right|,
\end{alignat*}

Throughout the rest of the argument we work with the following set
of representatives: Each equivalence class $[w]$ contains a unique
word $u$ with $\mathrm{supp}_{0}(u)=\{1,2,\ldots,|\mathrm{supp_{0}}(u)|\}$and
such that the appearance of the $0$-cells along the word $u$, is
in increasing order. We choose this word as the unique representative
of the equivalence class. Note that given such a representative $u$
for the equivalence class, the remaining elements in the equivalence
class are given via a permutation $\mathbf{v}\in[n]^{|\mathrm{supp}_{0}(u)|}$,
taking the word $u$ to the word $\mathbf{v}(u)=\mathbf{v}(u_{1})\mathbf{v}(u_{2})\cdots\mathbf{v}(u_{|\mathrm{supp}_{0}(u)|})$,
where we recall that for a cell $\sigma=[\sigma^{0},\sigma^{1},\ldots,\sigma^{d-1}]$
in $X^{d-1}$, we define $\mathbf{v}(\sigma)=[\mathbf{v}(\sigma^{0}),\mathbf{v}(\sigma^{1}),\ldots,\mathbf{v}(\sigma^{d-1})]$.
With a slight abuse of notation, we use $\left[n\right]^{m}$, for
$m\in\mathbb{N}$, to denote all vectors of length $m$, whose components
are distinct and belong to the set $\left[n\right]$. 
\begin{example}
For $d=2$ and $n=8$, consider the word $w=\left[5,6\right]\left[6,7\right]\left[5,6\right]$$\left[6,1\right]\left[1,2\right]$.
The unique representative in the equivalence class of $w$ is $u=[1,2][2,3][1,2][2,4][4,5]$.
Taking the permutation $\mathbf{v}=\left(6,4,8,1,5\right)\in\left[8\right]^{5}$
, gives the equivalent word $\mathbf{v}(u)=[6,4][4,8][6,4][4,1][1,5]$.
\end{example}

Let $w$ be a representative of an equivalence class. Observe that
each closed word $u$ satisfying $u\sim w$, arises from a unique
permutation $\mathbf{v}^{u}\in[n]^{|\mathrm{supp}_{0}(w)|}$. Hence,
\[
\left|\left\{ u:u\text{ a closed word such that }u\sim w\right\} \right|\leq\left|[n]^{|\mathrm{supp}_{0}(w)|}\right|.
\]
Consequently,
\begin{alignat}{1}
\mathbb{E}\left[\left\Vert H\right\Vert _{S_{2k}}^{2k}\right] & \leq\frac{1}{\left(n\text{Var}\left(Z\right)\right)^{k}}\sum_{w\in\mathcal{W}_{2k+1}}\prod_{\tau\in\text{supp}_{d}\left(w\right)}b_{\tau}^{\left(N_{\ensuremath{w}}\left(\tau\right)\right)}\left|[n]^{|\mathrm{supp}_{0}(w)|}\right|\nonumber \\
 & =\frac{1}{\left(n\text{Var}\left(Z\right)\right)^{k}}\sum_{w\in\mathcal{W}_{2k+1}}\sum_{\mathbf{v}\in[n]^{|\mathrm{supp}_{0}(w)|}}\prod_{\tau\in\text{supp}_{d}\left(w\right)}b_{\tau}^{\left(N_{\ensuremath{w}}\left(\tau\right)\right)}\nonumber \\
 & =\frac{1}{\left(n\text{Var}\left(Z\right)\right)^{k}}\sum_{w\in\mathcal{W}_{2k+1}}\sum_{\mathbf{v}\in[n]^{|\mathrm{supp}_{0}(w)|}}\prod_{\tau\in\text{supp}_{d}\left(w\right)}b_{\mathbf{v}\left(\tau\right)}^{\left(N_{\ensuremath{w}}\left(\tau\right)\right)},\label{eq:1}
\end{alignat}
where in the last equality we used the fact that $\left|B_{\mathbf{v}\left(\tau\right)}\right|\overset{\text{law}}{=}\left|B_{\tau}\right|$,
and thus $b_{\tau}^{\left(N_{\ensuremath{w}}\left(\tau\right)\right)}=b_{\mathbf{v}\left(\tau\right)}^{\left(N_{\ensuremath{w}}\left(\tau\right)\right)}$.\textbf{ }
\begin{notation}
\label{nota:Notations1}Given $G=\left(V_{G},E_{G}\right)\in\mathbb{G}$
we define
\[
S_{d}\left(G\right):=\left\{ \tau\in X^{d}\text{ : }\exists\left\{ \sigma,\sigma'\right\} \in E_{G}\text{ such that \ensuremath{\sigma\cup\sigma'=\tau}}\right\} ,
\]
 and 
\[
S_{0}\left(G\right)=\left\{ i\in\left[n\right]\text{ : }\exists\sigma\in V_{G}\text{ such that \ensuremath{i\in\sigma}}\right\} .
\]
Note that for a word $u$, $S_{d}\left(G_{u}\right)=\text{supp}_{d}\left(u\right)$
and $S_{0}\left(G_{u}\right)=\text{supp}_{0}\left(u\right)$. \\
Denote by $\mathbf{N}_{G}=\left(\mathcal{N}_{G}\left(e\right)\right)_{e\in E_{G}}$
a family of positive weights for the edges in $G$. For $\tau\in S_{d}\left(G\right)$
define $\mathcal{N}_{G}\left(\tau\right):=\sum_{\underset{\sigma\cup\sigma'=\tau}{\left\{ \sigma,\sigma'\right\} \in E_{G}}}\mathcal{N}_{G}\left(\left\{ \sigma,\sigma'\right\} \right).$
\end{notation}

\begin{notation}
\label{nota:Notations2}For a graph $G\in\mathbb{G}$ and weights
$\mathbf{N}_{G}=\left(\mathcal{N}_{G}\left(\tau\right)\right)_{\tau\in S_{d}\left(G\right)}$,
denote
\[
\mathcal{G}\left(G;\mathbf{N}_{G}\right)=\sum_{\mathbf{v}\in[n]^{|S_{0}\left(G\right)|}}\prod_{\tau\in S_{d}\left(G\right)}b_{\mathbf{v}\left(\tau\right)}^{\left(\mathcal{N}_{G}\left(\tau\right)\right)}.
\]

Note that $\mathcal{G}\left(G;\mathbf{N}_{G}\right)$ equals $\left|S_{0}\left(G\right)\right|!\cdot{n \choose \left|S_{0}\left(G\right)\right|}\prod_{\tau\in S_{d}\left(G\right)}b_{\tau}^{\left(\mathcal{N}_{G}\left(\tau\right)\right)},$
as $b_{\mathbf{v}\left(\tau\right)}^{\left(\mathcal{N}_{G}\left(\tau\right)\right)}$
does not depend on the choice of $\mathbf{v}$. Nevertheless, we keep
the original notation including the sum in order to apply later on
Hölder's inequality on it.
\end{notation}

Using Notations \ref{nota:Notations1} and \ref{nota:Notations2}
, we can write inequality \eqref{eq:1} as follows
\begin{equation}
\mathbb{E}\left[\left\Vert H\right\Vert _{S_{2k}}^{2k}\right]\leq\frac{1}{\left(n\text{Var}\left(Z\right)\right)^{k}}\sum_{w\in\mathcal{W}_{2k+1}}\mathcal{G}\left(G_{w};\mathbf{N}_{G_{w}}\right),\label{eq:2}
\end{equation}
where the weights $\mathbf{N}_{G_{w}}=\left(N_{w}\left(\tau\right)\right)_{\tau\in\text{supp}_{d}\left(w\right)}$,
are taken to be the crossing numbers.

\subsubsection{Reduction to the case of trees}

The following result is a generalization of \cite[Lemma 2.9]{LHY18},
showing that among all graphs $G\in\mathbb{G}$, the value $\mathcal{G}\left(G;\mathbf{N}_{G}\right)$
is maximized by trees. This will enable us to restrict attention to
trees in the rest of the proof. The main difference is that for $d\geq2$,
there is more than one way that a cycle in the induced graph can traversed
a $d$-cell (see Figures \ref{fig:Case1.1}, \ref{fig:Case1.2.1} and
\ref{fig:Case1.2.2} ), whereas for $d=1$ such a crossing is unique.
In particular, the number of $d-$cells which are crossed by a cycle
of length $\ell$ in $d=1$ is always $\ell$, while for $d\geq2$
it is only bounded from above by $\ell$, and in many cases it is
strictly smaller. The analysis corresponding to the only available
situation in the graph case, namely Figure \ref{fig:Case1.2.2}, is
similar to that in \cite[Lemma 2.9]{LHY18}. However, new arguments
are needed in order to deal with the new cases that does not exist
in the one-dimensional case.
\begin{lem}
\label{lem:Lemma1}For every word $w\in\mathcal{W}_{2k+1}$ and every
family of labelings $\mathbf{N}_{G_{w}}=\left(\mathcal{N}_{G_{w}}\left(e\right)\right)_{e\in E\left(G_{w}\right)}$
for the associated graph $G_{w}$, there exist a graph $T\in\mathbb{G}$
and labelings $\mathbf{N}_{T}=\left(\mathcal{N}_{T}\left(\tau\right)\right)_{\tau\in S_{d}\left(T\right)}$
(depending on $\mathbf{N}_{G_{w}}$), such that the following holds: 
\begin{enumerate}
\item $T$ is a tree.
\item $S_{d}\left(T\right)\subseteq S_{d}\left(G_{w}\right).$
\item $S_{0}\left(T\right)=S_{0}\left(G_{w}\right)$.
\item $\mathcal{N}_{T}\left(\tau\right)\geq\mathcal{N}_{G_{w}}\left(\tau\right),$$\forall\tau\in S_{d}\left(T\right)$.
\item $\sum_{\tau\in S_{d}\left(G_{w}\right)}\mathcal{N}_{G_{w}}\left(\tau\right)=\sum_{\tau\in S_{d}\left(T\right)}\mathcal{N}_{T}\left(\tau\right)$.
\item $\mathcal{G}\left(G_{w};\mathbf{N}_{G_{w}}\right)\leq\mathcal{G}\left(T;\mathbf{N}_{T}\right)$.
\end{enumerate}
\end{lem}

\begin{proof}
If $G_{w}$ is a tree, then by setting $T:=G_{w}$ and $\mathcal{N}_{T}\left(\tau\right)=\mathcal{N}_{G_{w}}\left(\tau\right)$
for every $\tau\in S_{d}\left(G\right)$ we are done. Next, assume
$G_{w}$ is not a tree, namely it contains a cycle. Denote such a
cycle by $e_{1}e_{2}\ldots e_{j}e_{1}$, where $e_{i}\in E_{G_{w}}$
for all $1\leq i\leq j$. There are three possible cases:
\begin{itemize}
\item \textbf{Case $1.1$:} The cycle is contained inside a $d$-cell. In
this case, we define a new graph, induced from $G_{w}$, by omitting
the edge $e_{1}$, and defining its new labeling via
\[
\mathcal{N}_{\text{new}}(e)=\begin{cases}
\mathcal{N}_{G_{w}}(e) & e\notin\{e_{1},e_{j}\}\\
\mathcal{N}_{G_{w}}(e_{1})+\mathcal{N}_{G_{w}}(e_{j}) & e=e_{j}
\end{cases}.
\]
See Figure \ref{fig:Case1.1} for an illustration in the case $d=2$.
\begin{figure}[h]
\begin{centering}
\includegraphics[scale=0.5]{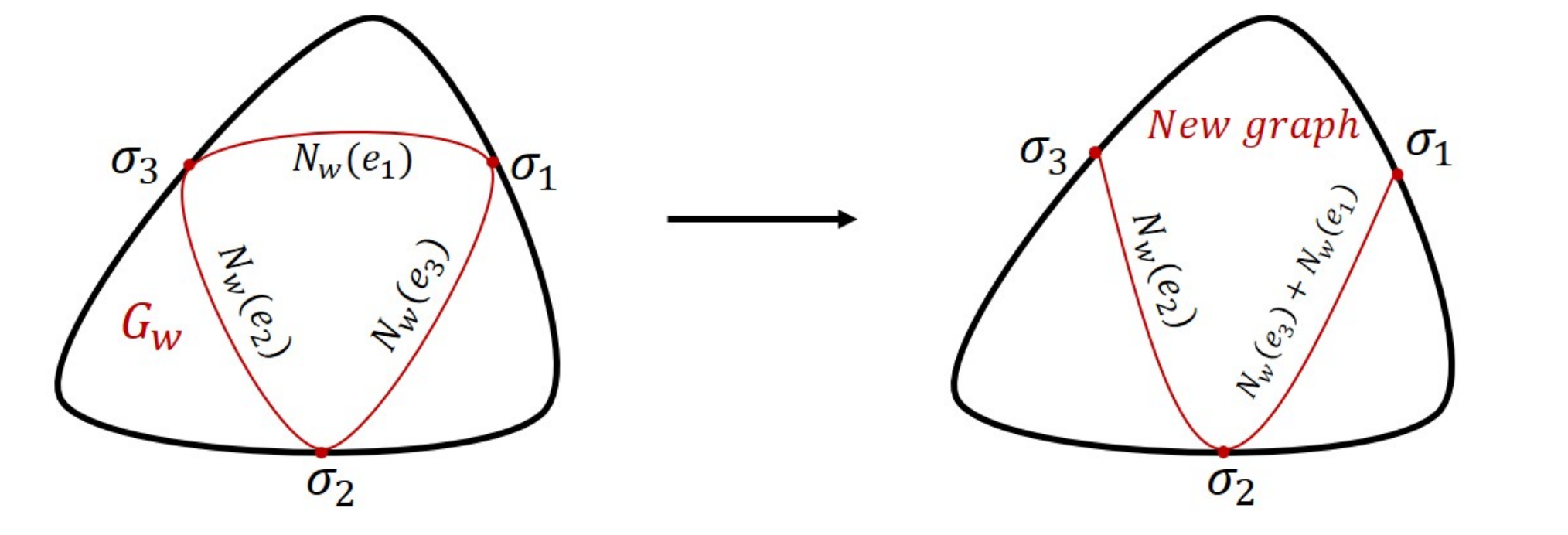}
\par\end{centering}
\caption{Left: A cycle inside a $2-$cell in the initial graph $G_{w}$. Right:
The transition of $G_{w}$ into a new graph, with its new labelings,
which does not contain the cycle\label{fig:Case1.1}}
\end{figure}
\item \textbf{Case $1.2$:} The cycle is not contained in a $d$-cell. We
observe two possible sub-cases:
\begin{itemize}
\item \textbf{Case $1.2.1$:} $\exists\tau\in S_{d}\left(G_{w}\right)$
and $\exists i_{1},i_{2}\in\left[j\right]$ distinct such that $e_{i_{1}},e_{i_{2}}$
cross $\tau$. We define a new graph, induced from $G_{w}$, by omitting
the edge $e_{i_{1}}$, and defining its new labeling via
\[
\mathcal{N}_{\text{new}}\left(e\right):=\begin{cases}
\mathcal{N}_{G_{w}}\left(e\right) & \ensuremath{e\notin\left\{ e_{i_{1}},e_{i_{2}}\right\} }\\
\mathcal{N}_{G_{w}}\left(e_{i_{1}}\right)+\mathcal{N}_{G_{w}}\left(e_{i_{2}}\right) & e=e_{i_{2}}
\end{cases},
\]
see Figure \ref{fig:Case1.2.1} for an illustration in the case $d=2$.
\begin{figure}[h]
\begin{centering}
\includegraphics[scale=0.5]{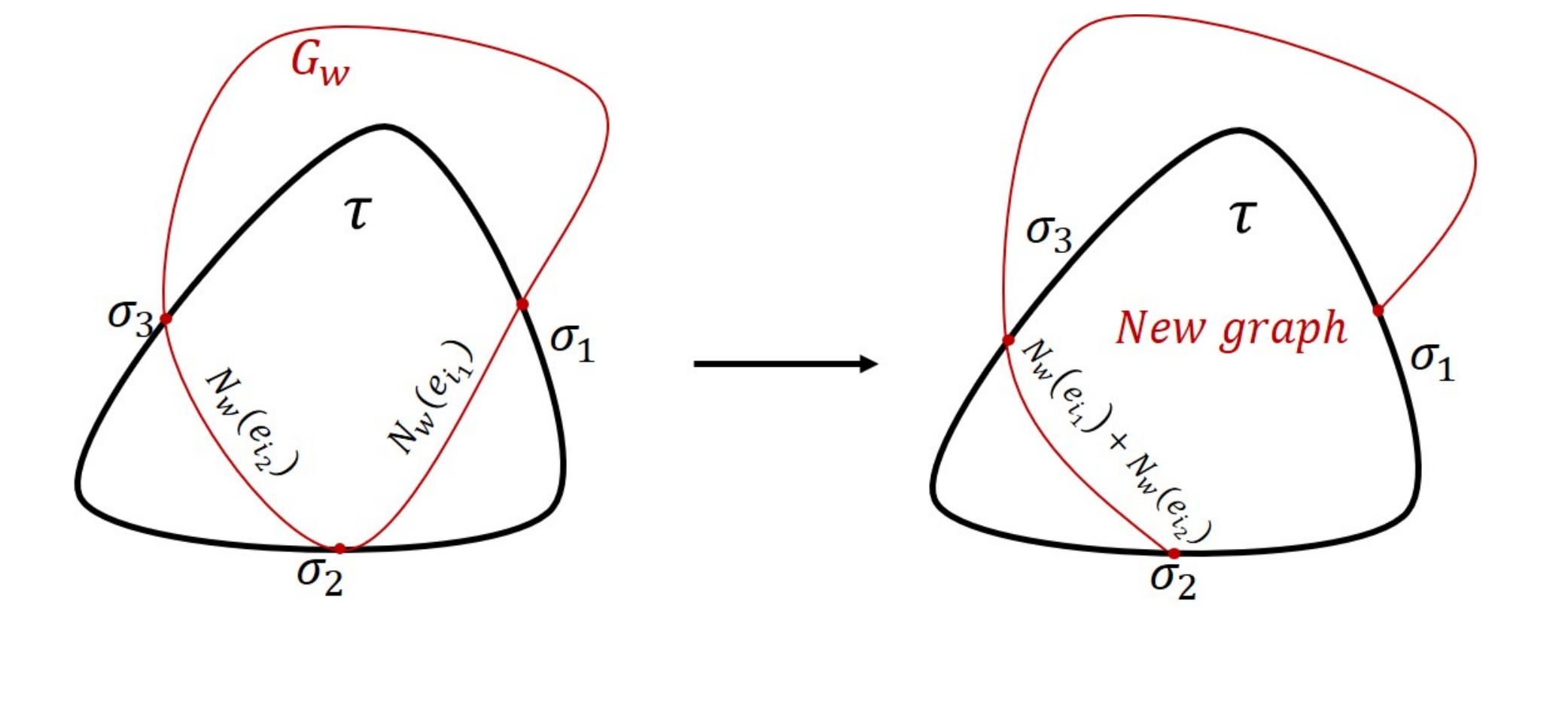}
\par\end{centering}
\caption{Left: A $2-$cell crossed by two edges of a cycle in $G_{w}$. Right:
The transition of $G_{w}$ into a new graph, with its new labelings,
which does not contain the cycle.\protect \\
\emph{Remark:} Note that for $d\protect\geq3$, $e_{i_{1}}$ and $e_{i_{2}}$
not necessarily have a common vertex, but this does not alter the
proof.\label{fig:Case1.2.1}}
\end{figure}
\item \textbf{Case $1.2.2$:} For any $d-$cell $\tau\in S_{d}\left(G_{w}\right)$,
crossed by the cycle, there exists a unique $i_{\tau}\in\left[j\right]$,
such that $e_{i_{\tau}}$ crosses $\tau$. Let $\tau_{1},\tau_{2}$
be two distinct $d-$cells that are both traversed by the cycle (the
existence of such cells is guaranteed by the above assumption), with
a common $\left(d-1\right)-$cell. Denote by $e_{i_{\tau_{1}}}$ and
$e_{i_{\tau_{2}}}$, with $i_{\tau_{1}},i_{\tau_{2}}\subset\left[j\right]$,
the unique edges of the cycle, crossing $\tau_{1}$ and $\tau_{2}$
respectively. Then by Jensen's inequality 
\begin{alignat*}{1}
\mathcal{G}\left(G_{w};\mathbf{N}_{G_{w}}\right) & =\sum_{\mathbf{v}\in[n]^{|S_{0}\left(w\right)|}}\prod_{i=1}^{2}b_{\mathbf{v}\left(\tau_{i}\right)}^{\left(\mathcal{N}_{\ensuremath{G_{w}}}\left(\tau_{i}\right)\right)}\prod_{\tau\in S_{d}\left(G_{w}\right)\backslash\left\{ \tau_{1},\tau_{2}\right\} }b_{\mathbf{v}\left(\tau\right)}^{\left(\mathcal{N}_{\ensuremath{G_{w}}}\left(\tau\right)\right)}\\
 & \leq\sum_{\mathbf{v}\in[n]^{|S_{0}\left(w\right)|}}\prod_{i=1}^{2}\left(b_{\mathbf{v}\left(\tau_{i}\right)}^{\left(\sum_{j=1}^{2}\mathcal{N}_{\ensuremath{G_{w}}}\left(\tau_{j}\right)\right)}\right)^{\frac{\mathcal{N}_{\ensuremath{G_{w}}}\left(\tau_{i}\right)}{\sum_{j=1}^{2}\mathcal{N}_{\ensuremath{G_{w}}}\left(\tau_{j}\right)}}\prod_{\tau\in S_{d}\left(G_{w}\right)\backslash\left\{ \tau_{1},\tau_{2}\right\} }b_{\mathbf{v}\left(\tau\right)}^{\left(\mathcal{N}_{G_{w}}\left(\tau\right)\right)}\\
 & =\sum_{\mathbf{v}\in[n]^{|S_{0}\left(w\right)|}}\prod_{i=1}^{2}\left(b_{\mathbf{v}\left(\tau_{i}\right)}^{\left(\sum_{j=1}^{2}\mathcal{N}_{\ensuremath{G_{w}}}\left(\tau_{j}\right)\right)}\cdot\prod_{\tau\in S_{d}\left(G_{w}\right)\backslash\left\{ \tau_{1},\tau_{2}\right\} }b_{\mathbf{v}\left(\tau\right)}^{\left(\mathcal{N}_{G_{w}}\left(\tau\right)\right)}\right)^{\frac{\mathcal{N}_{G_{w}}\left(\tau_{i}\right)}{\sum_{j=1}^{2}\mathcal{N}_{\ensuremath{G_{w}}}\left(\tau_{j}\right)}}\\
 \end{alignat*}
Since $b_{\tau}^{\left(\alpha\right)}=b_{\tau'}^{\left(\alpha\right)}$ for all $\tau,\tau'\in K^{d}$, it follows from Hölder's inequality that 
\begin{alignat*}{1}
 \mathcal{G}\left(G_{w};\mathbf{N}_{G_{w}}\right) &\leq  \prod_{i=1}^{2}\left(\sum_{\mathbf{v}\in[n]^{|S_{0}\left(w\right)|}}b_{\mathbf{v}\left(\tau_{i}\right)}^{\left(\sum_{j=1}^{2}\mathcal{N}_{\ensuremath{G_{w}}}\left(\tau_{j}\right)\right)}\cdot\prod_{\tau\in S_{d}\left(G_{w}\right)\backslash\left\{ \tau_{1},\tau_{2}\right\} }b_{\mathbf{v}\left(\tau\right)}^{\left(\mathcal{N}_{\ensuremath{G_{w}}}\left(\tau\right)\right)}\right)^{\frac{\mathcal{N}_{\ensuremath{G_{w}}}\left(\tau_{i}\right)}{\sum_{j=1}^{2}\mathcal{N}_{\ensuremath{G_{w}}}\left(\tau_{j}\right)}}\\
 & =\prod_{i=1}^{2}\left(\sum_{\mathbf{v}\in[n]^{|S_{0}\left(w\right)|}}b_{\mathbf{v}\left(\tau_{1}\right)}^{\left(\sum_{j=1}^{2}\mathcal{N}_{\ensuremath{G_{w}}}\left(\tau_{j}\right)\right)}\cdot\prod_{\tau\in S_{d}\left(G_{w}\right)\backslash\left\{ \tau_{1},\tau_{2}\right\} }b_{\mathbf{v}\left(\tau\right)}^{\left(\mathcal{N}_{\ensuremath{G_{w}}}\left(\tau\right)\right)}\right)^{\frac{\mathcal{N}_{\ensuremath{G_{w}}}\left(\tau_{i}\right)}{\sum_{j=1}^{2}\mathcal{N}_{\ensuremath{G_{w}}}\left(\tau_{j}\right)}}\\
 & =\sum_{\mathbf{v}\in[n]^{|S_{0}\left(w\right)|}}b_{\mathbf{v}\left(\tau_{1}\right)}^{\left(\sum_{j=1}^{2}\mathcal{N}_{\ensuremath{G_{w}}}\left(\tau_{j}\right)\right)}\cdot\prod_{\tau\in S_{d}\left(G_{w}\right)\backslash\left\{ \tau_{1},\tau_{2}\right\} }b_{\mathbf{v}\left(\tau\right)}^{\left(\mathcal{N}_{\ensuremath{G_{w}}}\left(\tau\right)\right)},
\end{alignat*}

Therefore we can define a new graph $G^{\left(\text{new}\right)}$
by omitting the unique edge which crosses $\tau_{2}$, namely $e_{i_{\tau_{2}}}$,
and define the graph labelings via
\begin{gather*}
\mathcal{N}_{\text{new}}\left(e\right):=\begin{cases}
\mathcal{N}_{G_{w}}\left(e\right) & e\notin\left\{ e_{i_{\tau_{1}}},e_{i_{\tau_{2}}}\right\} \\
\mathcal{N}_{G_{w}}\left(e_{i_{\tau_{1}}}\right)+\mathcal{N}_{G_{w}}\left(e_{i_{\tau_{2}}}\right) & e=e_{i_{\tau_{1}}}
\end{cases}.
\end{gather*}
The above computation shows that the new graph together with its new
labeling, satisfies all the requirements described in Lemma \ref{lem:Lemma1}.
See Figure \ref{fig:Case1.2.2} for an illustration in the case $d=2$.
\begin{figure}[h]
\begin{centering}
\includegraphics[scale=0.5]{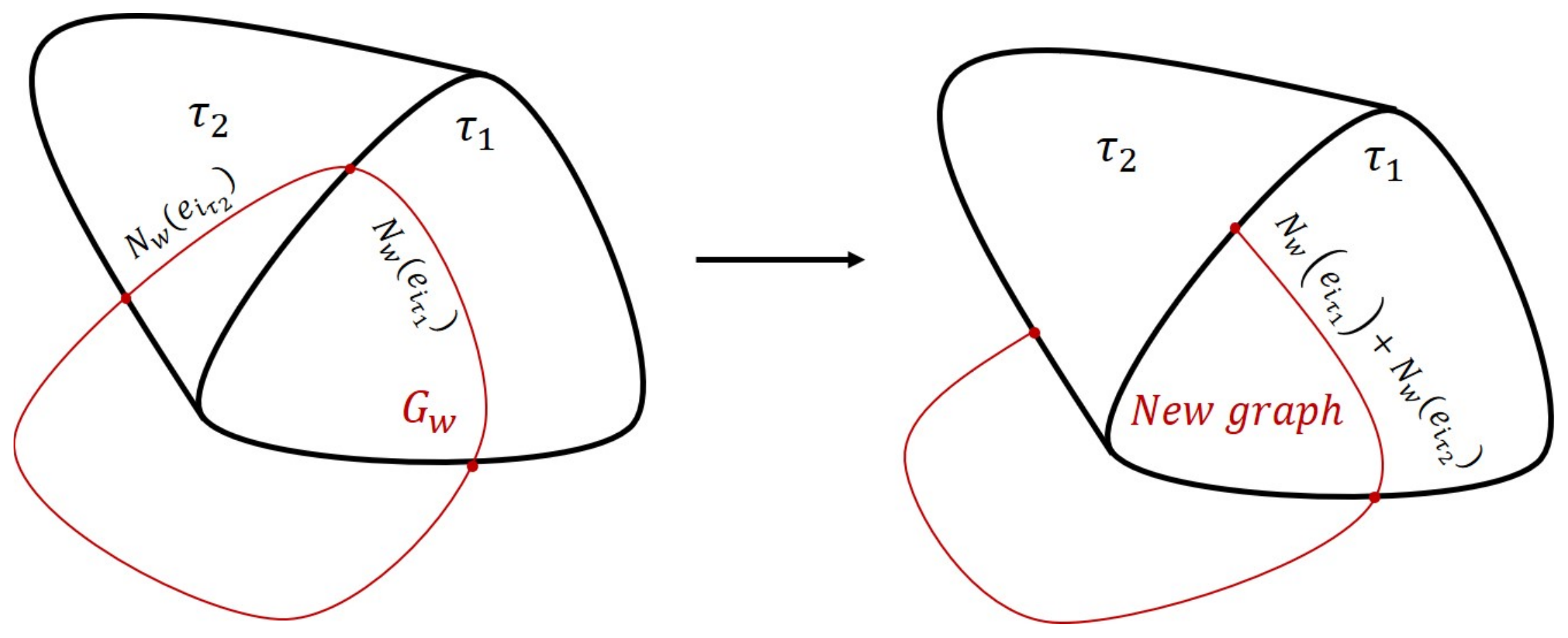}
\par\end{centering}
\centering{}\caption{Left: A cycle in $G_{w}$ whose traversed $d-$cells are crossed by
exactly one edge. Right: The transition of $G_{w}$ into a new graph,
with its new labelings, which does not contain the cycle.\label{fig:Case1.2.2}}
\end{figure}
\end{itemize}
\end{itemize}
Note that in all cases, the new graph attained has the same set of
vertices as $G_{w}$, and so $S_{0}\left(G_{w}\right)=S_{0}\left(G^{\left(\text{new}\right)}\right)$.
Moreover, since we merely omit edges which are part of a cycle, the
new graph remains connected. Yet, the new graph contains one less
cycles than the original graph. 

We repeat the above procedure on the graph $G^{\left(\text{new}\right)}$
repeatedly until there are no cycles left. Denote by $T$ a graph
attained via this process, and by $\mathbf{N}_{T}:=\left(\mathcal{N}_{T}\left(\tau\right)\right)_{\tau\in S_{d}\left(T\right)}$
the resulting labelings of $S_{d}\left(T\right)$ . 

Since $T$ is connected and does not contain any cycle, it is a tree
and belongs to $\mathbb{G}$: denote $V_{T}:=$$\left\{ \sigma_{i}\right\} _{i=1}^{N}$
. Since $T$ is connected, between any $i\neq j$ there is a path,
and thus a word which generates it, denoted by $w_{i,j}$. Assigning
all words $w_{1,2}w_{2,3}\ldots w_{N-1,N}$ gives a new word which
induces the graph $T$ . Furthermore, it is clear from the construction
of $T$ that condition $\left(2\right)-\left(6\right)$ are satisfied,
thus concluding the proof.
\end{proof}
The following lemma is a generalization of \cite[Lemma 2.10]{LHY18}.
As in \cite[Lemma 2.10]{LHY18}, we address the leaves of the tree,
and use similar ideas in order to apply Hölder's inequality recursively.
The main difference is that for $d=1$, the authors in \cite[Lemma 2.10]{LHY18}
used the following property, which is no longer true for $d\geq2$:
A crossing of a $1$-cell is in bijection with the pair of vertices
in its boundary. Namely, given a tree $T,$ and a leaf $\sigma=\left\{ i\right\} $,
there exists a unique $1-$cell $\tau=\left\{ i,j\right\} $ that
contains $i$ and the only possible crossing from $i$ to the remainder
of the tree is the cross from $i$ to $j$. To overcome the new phenomena
in the high-dimensional setting, we use additional combinatorial arguments
which allow us to complete the argument similarly to \cite[Lemma 2.10]{LHY18}.
\begin{lem}
\label{Lem:lemma_2}For any word $u$ such that $G_{u}$ is a tree
and every labelings of its edges $\mathbf{N}_{G_{u}}=\left(\mathcal{N}_{G_{u}}\left(e\right)\right)_{e\in E\left(G_{u}\right)}$,
there exists $S\subseteq S_{d}\left(G_{u}\right)$ and $\text{ \ensuremath{\left(\mathcal{M}\left(\tau\right)\right)_{\tau\in S}}}$
such that the following holds:
\begin{enumerate}
\item $\left|S\right|=\left|S_{0}\left(G_{u}\right)\right|-d$.
\item $\mathcal{M}\left(\tau\right)\geq\mathcal{N}_{\ensuremath{G_{u}}}\left(\tau\right)$
for all $\tau\in S$.
\item $\sum_{\tau\in S}\mathcal{M}\left(\tau\right)=\sum_{\tau\in S_{d}\left(G_{u}\right)}\mathcal{N}_{G_{u}}\left(\tau\right)$.
\end{enumerate}
Furthermore, given any $\left(p_{\tau}\right)_{\tau\in S}\subset(0,1]$
such that $\sum_{\tau\in S}\frac{1}{p_{\tau}}=1$ 

\begin{equation}
\mathcal{G}\left(G_{u};\mathbf{N}_{G_{u}}\right)\leq d!\prod_{\tau\in S}\left(\sum_{\sigma\in X_{+}^{d-1}}\left(\sum_{i\in\left[n\right],i\notin\sigma}b_{\sigma\omega^{\left(\sigma,i\right)}}^{\left(\mathcal{M}\left(\tau\right)\right)}\right)^{p_{\tau}}\right)^{\frac{1}{p_{\tau}}},\label{eq:3}
\end{equation}
 where $\omega^{\left(\sigma,i\right)}\in\Sigma_{\sigma,i}:=\left\{ \left.\sigma'\in X_{+}^{d-1}\right|i\in\sigma'\text{ and \ensuremath{\sigma\overset{K^{d}}{\sim}\sigma'} or \ensuremath{\sigma\overset{K^{d}}{\sim}\overline{\sigma'}}}\right\} $.

\end{lem}

\begin{rem*}
The expression on the right hand side of \eqref{eq:3} is independent
of the choice of $\omega^{\left(\sigma,i\right)}$, since $b_{\sigma\sigma'}^{\left(\mathcal{M}\left(\tau\right)\right)}=b_{\sigma\sigma''}^{\left(\mathcal{M}\left(\tau\right)\right)}$
for any $\sigma'$,$\sigma''\in\Sigma_{\sigma,i}$.
\end{rem*}
\begin{proof}
The proof proceeds by induction on $\left|S_{0}\left(G_{u}\right)\right|$.
We operate a procedure on the graph $G_{u}$ in order to obtain a
new graph which is still a tree, by omitting $0$-cells from the original
graph. This allows us to use the induction hypothesis on the smaller
graph. We describe how the omitted $0$-cell is chosen, and we show
there is a 1-1 correspondence between an omitted $0$-cell and an
omitted $d$-cell. The set $S$ is the set of all omitted $d$-cells
according to this correspondence. We simultaneously prove inequality
\eqref{eq:3}, by an additional induction, where in each step we have
an omitted $d$-cell $\tau$, and we attach to it a number $p_{\tau}\in(0,1]$.
Observe that the set $S$, which is determined using the mentioned
procedure, is independent with the proof of inequality \eqref{eq:3},
and we do it together for the sake of simplicity and coherence. 

For the initial case, if $\left|S_{0}\left(G_{u}\right)\right|=d+1,$
then $\left|S_{d}\left(G_{u}\right)\right|=1$, namely $S_{d}\left(G_{u}\right)=\{\tau_{0}\}$
for some $\tau_{0}\in X^{d}$, and the result follows readily by setting
$S=\{\tau_{0}\}$ and $\mathcal{M}\left(\tau_{0}\right):=\mathcal{N}_{G_{u}}\left(\tau_{0}\right)$.
Indeed, it is clear that conditions $\left(1\right)-\left(3\right)$
are satisfied. Furthermore, given $p_{\tau}>0$ with $\sum_{\tau\in S}\frac{1}{p_{\tau}}=1$,
since $S=\left\{ \tau_{0}\right\} $ it follows that $p_{\tau_{0}}=1$,
and therefore
\begin{alignat*}{1}
\mathcal{G}\left(G_{u};\mathbf{N}_{G_{u}}\right) & =\sum_{\mathbf{v}\in[n]^{|S_{0}\left(G_{u}\right)|}}\prod_{\tau\in S_{d}\left(G_{u}\right)}b_{\mathbf{v}\left(\tau\right)}^{\left(\mathcal{N}_{G_{u}}\left(\tau\right)\right)}=\sum_{\mathbf{v}\in[n]^{d+1}}b_{\mathbf{v}\left(\tau_{0}\right)}^{\left(\mathcal{M}\left(\tau_{0}\right)\right)}\\
 & =\sum_{\mathbf{v}\in[n]^{d}}\sum_{i\in\left[n\right]\backslash\mathbf{v}\left(\sigma_{0}\right)}b_{i\cup\mathbf{v}\left(\sigma_{0}\right)}^{\left(\mathcal{M}\left(\tau_{0}\right)\right)}\\
 & \overset{\left(1\right)}{=}d!\sum_{\sigma\in X_{+}^{d-1}}\left(\sum_{i\in\left[n\right],i\notin\sigma}b_{\sigma\omega^{\left(\sigma,i\right)}}^{\left(\mathcal{M}\left(\tau_{0}\right)\right)}\right),
\end{alignat*}
where $\left(1\right)$ follows because each $\sigma\in X_{+}^{d-1}$
is attained by $d!$ different $\mathbf{v}\in[n]^{d}$ (up to reordering
of its entries). This concludes the proof in the case $\left|S_{0}\left(G_{u}\right)\right|=d+1$.

Let us now describe the induction step. Suppose we have shown that
for some $r\in\mathbb{N}$, whenever $G_{u}$ is a tree with $|S_{0}(G_{u})|>d+r$,
one can find $S=\left\{ \tau_{s}\right\} _{s=1}^{r}\subseteq S_{d}\left(G_{u}\right)$,
weights $\text{ \ensuremath{\left(\mathcal{M}\left(\tau_{s}\right)\right)_{s=1}^{r}}}$,
a graph $G\in\mathbb{G}$ and weights $\left(\mathcal{N}_{G}\left(\tau\right)\right)_{\tau\in S_{d}\left(G\right)}$
, such that 
\begin{alignat}{1}
\mathcal{G}\left(G_{u};\mathbf{N}_{G_{u}}\right) & \leq\prod_{s=1}^{\left|S\right|}\left[\sum_{\mathbf{v}\in[n]^{|S_{0}\left(G\right)|}}\left(\sum_{i\notin\mathbf{v}\left(\sigma_{1}\right)}b_{i\cup\mathbf{v}\left(\sigma_{1}\right)}^{\left(\mathcal{M}\left(\tau_{s}\right)\right)}\right)^{q_{s}}\prod_{\tau\in S_{d}\left(G\right)}b_{\mathbf{v}\left(\tau\right)}^{\left(\mathcal{N}_{G}\left(\tau\right)\right)}\right]^{\frac{1}{\alpha_{s}}},\label{eq:12}
\end{alignat}
where
\begin{enumerate}
\item $G$ is a tree with
\begin{enumerate}
\item $\sigma_{1}\in X_{+}^{d-1}$ satisfies $\sigma_{1}\in E\left(G\right)$.
\item $S_{d}\left(G\right)\subseteq S_{d}\left(G_{u}\right)$.
\item $S_{0}\left(G\right)\subseteq S_{0}\left(G_{u}\right)$ with $\left|S_{0}\left(G\right)\right|=\left|S_{0}\left(G_{u}\right)\right|-r$.
\end{enumerate}
\item The weights satisfy
\begin{enumerate}
\item $\mathcal{M}\left(\tau\right)\geq\mathcal{N}_{\ensuremath{G_{u}}}\left(\tau\right)$
for all $\tau\in S$.
\item $\sum_{\tau\in S}\mathcal{M}\left(\tau\right)=\sum_{\tau\in S_{d}\left(G_{u}\right)}\mathcal{N}_{G_{u}}\left(\tau\right)$.
\item $\mathcal{N}_{G}\left(\tau\right)\geq\mathcal{N}_{G_{\mathbf{u}}}\left(\tau\right)$
for any $\tau\in S_{d}\left(G\right)$.
\item $\sum_{\tau\in S_{d}\left(G_{\mathbf{u}}\right)}\mathcal{N}_{G_{u}}\left(\tau\right)=\sum_{s=1}^{r}\mathcal{M}\left(\tau_{s}\right)+\sum_{\tau\in S_{d}\left(G\right)}\mathcal{N}_{G}\left(\tau\right)$.
\end{enumerate}
\item The numbers $q_{s}$ satisfy
\begin{enumerate}
\item $q_{s}=\sum_{j=1}^{r}\frac{p_{\tau_{s}}}{p_{\tau_{j}}}$ for all $1\leq s\leq r$,
where $(p_{\tau})_{\tau\in S}\subset(0,1]$.
\end{enumerate}
\item The right-hand side of \eqref{eq:12} is $1$-homogeneous in all variables
$\left\{ b^{\left(\mathcal{N}\left(\tau\right)\right)}\right\} $
which determine the value of $\alpha_{s}$.
\end{enumerate}
By the induction hypothesis $G$ is a tree and thus must contain a
leaf, i.e., vertex of degree one. Denote by $\sigma_{\ell}$ such
a leaf and by $\sigma'_{\ell}$ the unique $\left(d-1\right)-$cell
such that $\left\{ \sigma_{\ell},\sigma'_{\ell}\right\} $ is an edge.
Set $\tau_{\ell}:=\sigma_{\ell}\cup\sigma'_{\ell}$ and let $i_{\ell}\in S_{0}\left(G\right)$
be the unique vertex ($0-$cell) in $G$ such that $\tau_{\ell}=\sigma'_{\ell}\cup\{i_{\ell}\}$.
Denote by $\hat{\text{\ensuremath{\mathbb{G}}}}$ a subset of $\mathbb{G}$
which contains merely elements from $\mathbb{G}$ which are trees,
and 
\[
\Psi=\bigcup_{k=1}^{{n \choose d}}\left\{ \left(m_{j}\right)_{j\in I};\left|I\right|=k,\text{ \ensuremath{m_{j}}}\in\mathbb{N}\right\} ,
\]
and define the function $f:\hat{\text{\ensuremath{\mathbb{G}}}}\times\Psi\to\hat{\text{\ensuremath{\mathbb{G}}}}\times\Psi$
(corresponds to a tree $G$ and the labelings of its edges $\left(\mathcal{N}_{G}\left(e\right)\right)_{e\in E\left(G\right)})$
according to the following three possible cases
\begin{itemize}
\item \textbf{Case $2.1$: }$\tau_{\ell}$ is traversed only by the edge
$\left\{ \sigma_{\ell},\sigma_{i_{\ell}}\right\} $ and $\left\{ \tau\in S_{d}\left(G_{u}\right)|i_{\ell}\in\tau\right\} =\left\{ \tau_{\ell}\right\} $
(namely, the $0-$cell $i_{\ell}$ is contained only in $\tau_{\ell}).$
In this case $f\left(\left(G,\mathbf{N}_{G}\right)\right)=\left(G,\mathbf{N}_{G}\right)$
(namely no changes are done in $G$ and its labelings). See Figure \ref{fig:Case2.1}
for an illustration in the case $d=2$. \\
\begin{figure}[h]
\begin{centering}
\includegraphics[scale=0.5]{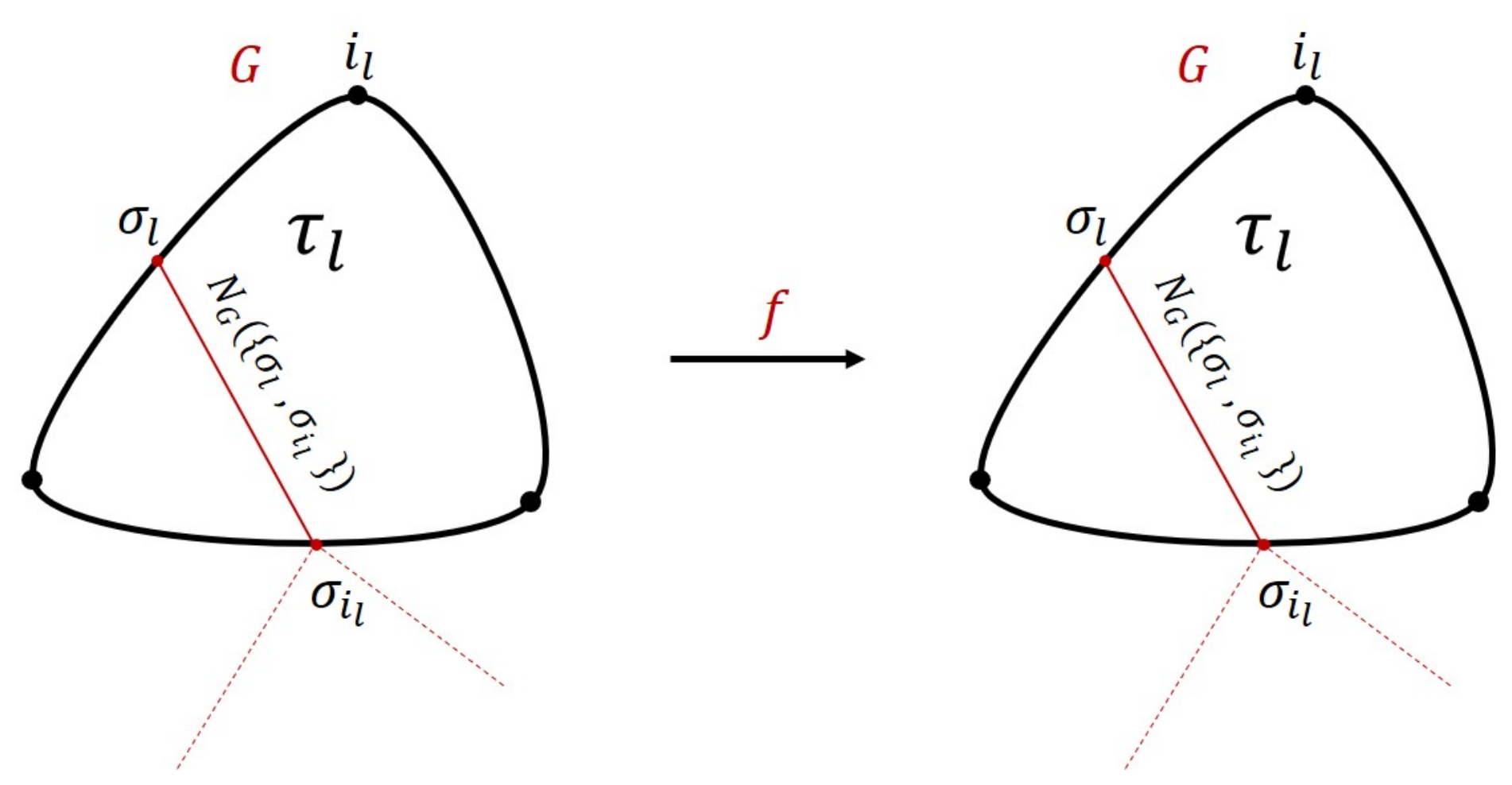}
\par\end{centering}
\caption{A leaf $\sigma_{\ell}$, in a $2-$cell $\tau_{\ell}$, such that
$i_{\ell}$ is contained only in the $2-$cell $\tau_{\ell}$ and
$\left\{ \sigma_{\ell},\sigma_{i_{\ell}}\right\} $ is the unique
edge which crosses $\tau_{\ell}$.\label{fig:Case2.1}}
\end{figure}
\item \textbf{Case $2.2$: }$\tau_{\ell}$ is traversed by an edge, $\left\{ \sigma_{j},\sigma_{m}\right\} $,
distinct from $\left\{ \sigma_{\ell},\sigma_{i_{\ell}}\right\} $.
In this case $f\left(\left(G,\mathbf{N}_{G}\right)\right)=\left(G_{1},\mathbf{N}_{G_{1}}\right)$,
where $G_{1}$ is the graph obtained from $G$ by omitting the edge
$\left\{ \sigma_{\ell},\sigma_{i_{\ell}}\right\} $, and labelings
for its edges $\mathbf{N}{}_{G_{1}}=\left(\mathcal{N}_{G_{1}}\left(e\right)\right)_{e\in E\left(G\right)\backslash\left\{ \sigma_{\ell},\sigma_{i_{\ell}}\right\} }$,
defined via: 
\begin{gather*}
\mathcal{N}_{G_{1}}\left(e\right):=\begin{cases}
\mathcal{N}_{G}\left(e\right) & e\notin\left\{ \left\{ \sigma_{j},\sigma_{m}\right\} ,\left\{ \sigma_{\ell},\sigma_{i_{\ell}}\right\} \right\} \\
\mathcal{N}_{G}\left(\left\{ \sigma_{\ell},\sigma_{i_{\ell}}\right\} \right)+\mathcal{N}_{G}\left(\left\{ \sigma_{j},\sigma_{m}\right\} \right) & e=\left\{ \sigma_{j},\sigma_{m}\right\} 
\end{cases},
\end{gather*}
See Figure \ref{fig:Case2.2} for an illustration in the case $d=2$.\\
\begin{figure}[h]
\begin{centering}
\includegraphics[scale=0.6]{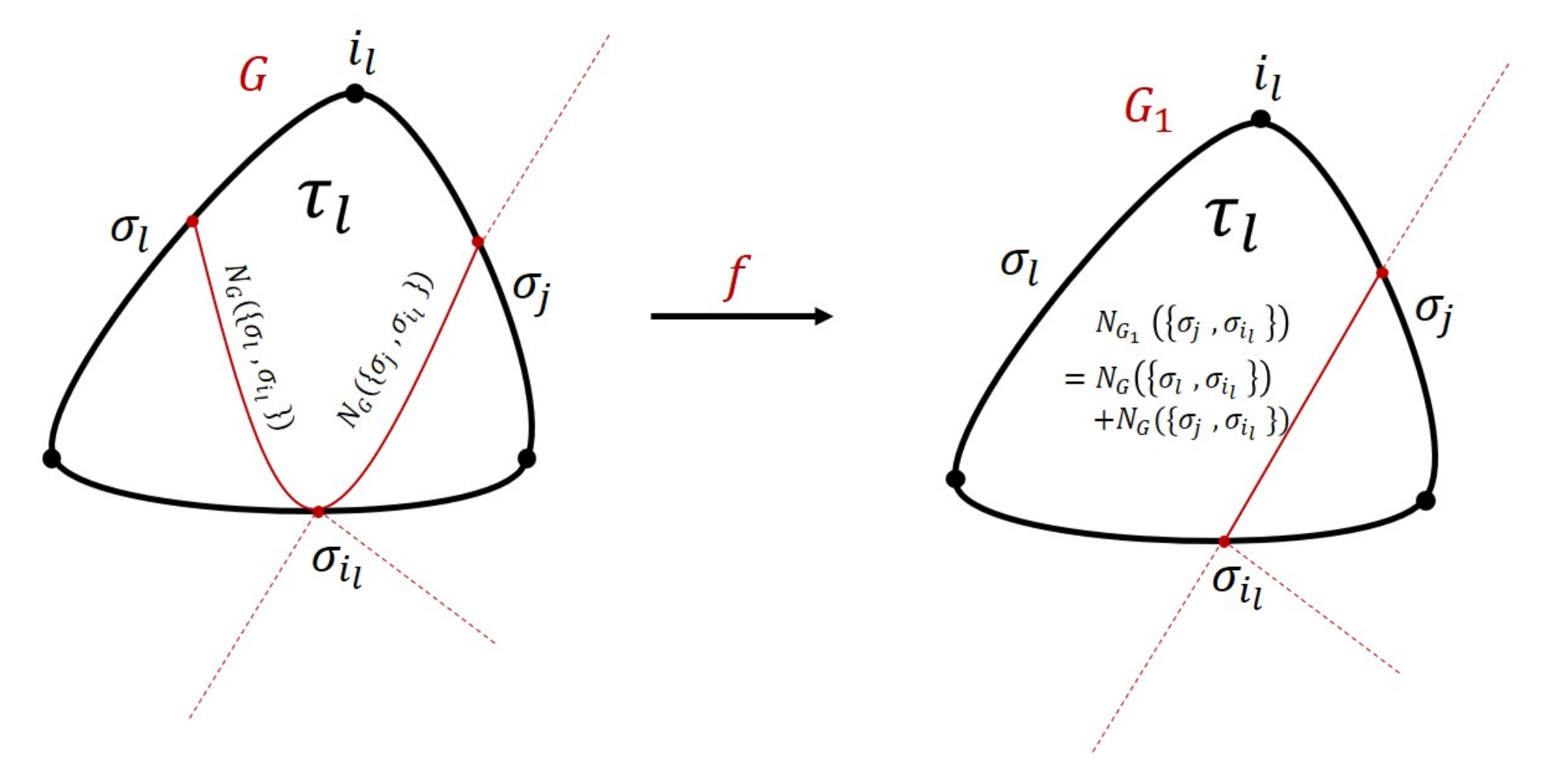}
\par\end{centering}
\caption{Left: A leaf $\sigma_{\ell}$, in a $2-$cell $\tau_{\ell}$, such
that $\tau_{\ell}$ is crossed by more than one edge ($\left\{ \sigma_{\ell},\sigma_{i_{\ell}}\right\} $
and $\left\{ \sigma_{j},\sigma_{m}\right\} $ where $m=i_{\ell}$).
Right: The transition of $G$ under $f$ into a new graph, $G_{1}$
, with new labelings for its edges.\label{fig:Case2.2}}
\end{figure}

\begin{itemize}
\item \textbf{Case $2.3$: }$\tau_{\ell}$ is traversed only by the edge
$\left\{ \sigma_{\ell},\sigma_{i_{\ell}}\right\} $, and $i_{\ell}$
is contained in a $d-$cell, $\tau\in S_{d}\left(G_{u}\right)$, distinct
from $\tau_{\ell}$. This case guarantees the existence of a $\sigma\in X_{+}^{d-1}$,
such that $\left\{ \sigma_{i_{\ell}},\sigma\right\} \in E\left(G_{u}\right)$
and $\sigma_{i_{\ell}}\cup\sigma=\tau'\in X^{d}$, where $\tau'\neq\tau_{\ell}$
(otherwise $G$ would have at least two connected components, in contradiction
to its connectivity). If there is more then one $\sigma\in X_{+}^{d-1}$
satisfying this condition, we choose one in a deterministic way and
define $f\left(\left(G,\mathbf{N}_{G}\right)\right)=\left(G_{2},\mathbf{N}_{G_{2}}\right)$,
where $G_{2}$ is the graph obtained from $G$ by omitting the edge
$\left\{ \sigma_{\ell},\sigma_{i_{\ell}}\right\} $, and we equip
$G_{2}$ with new labelings for its edges, $\mathbf{N}_{G_{2}}=\left(\mathcal{N}_{G_{2}}\left(e\right)\right)_{e\in E\left(G\right)\backslash\left\{ \sigma_{\ell},\sigma_{i_{\ell}}\right\} }$,
defined via:
\begin{gather*}
\mathcal{N}_{G_{2}}\left(e\right):=\begin{cases}
\mathcal{N}_{G}\left(e\right) & e\notin\left\{ \left\{ \sigma,\sigma_{i_{\ell}}\right\} ,\left\{ \sigma_{\ell},\sigma_{i_{\ell}}\right\} \right\} \\
\mathcal{N}_{G}\left(\left\{ \sigma_{\ell},\sigma_{i_{\ell}}\right\} \right)+\mathcal{N}_{G}\left(\left\{ \sigma,\sigma_{i_{\ell}}\right\} \right) & e=\left\{ \sigma,\sigma_{i_{\ell}}\right\} 
\end{cases},
\end{gather*}
See Figure \ref{fig:Case2.3} for illustration in the case $d=2$.\\
\begin{figure}[h]
\begin{centering}
\includegraphics[scale=0.5]{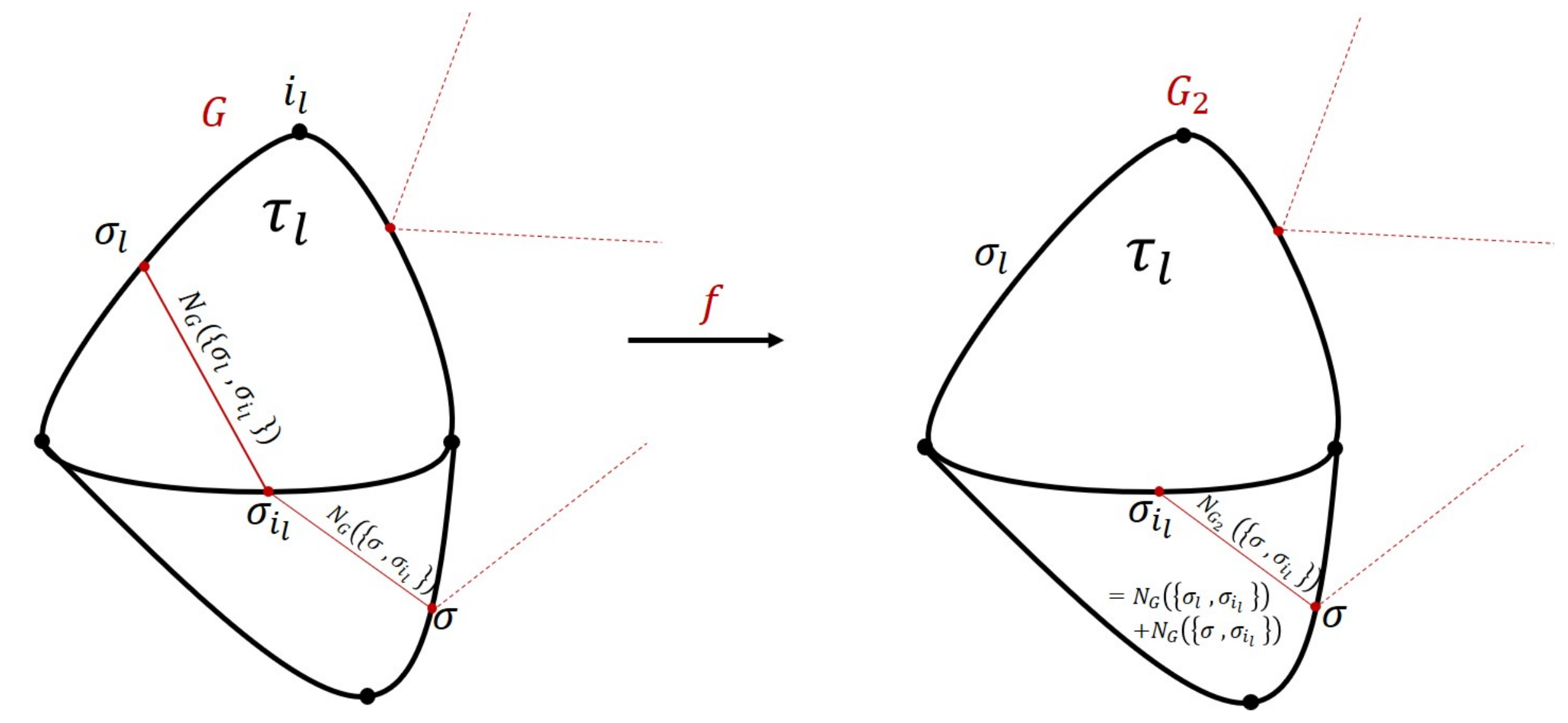}
\par\end{centering}
\caption{Left: A leaf $\sigma_{\ell}$ in a $2-$cell $\tau_{\ell}$ such that
$\tau_{\ell}$ is crossed uniquely by the edge $\left\{ \sigma_{\ell},\sigma_{i_{\ell}}\right\} $,
and $i_{\ell}$ is contained in $\tau\in S_{d}\left(G_{u}\right)$
such that $\tau\protect\neq\tau_{\ell}$. Right: The transition of
$G$ under $f$ into a new graph, $G_{2}$ , with the new labelings
for its edges.\label{fig:Case2.3}}
\end{figure}
\end{itemize}
It is clear that $\mathcal{G}\left(G;\mathbf{N}_{G}\right)=\mathcal{G}\left(G_{1};\mathbf{N}_{G_{1}}\right)$
since $S_{0}\left(G\right)=S_{0}\left(G_{1}\right)$, $S_{d}\left(G\right)=S_{d}\left(G_{1}\right)$
and the labelings of each $\tau\in S_{d}\left(G\right)$ are the same
as those of $G_{1}$. The inequality $\mathcal{G}\left(G;\mathbf{N}_{G}\right)\leq\mathcal{G}\left(G_{2};\mathbf{N}_{G_{2}}\right)$
follows from the same arguments which were introduced in Lemma \eqref{lem:Lemma1},
Case $1.2.2$. Observe that for $i\in\left\{ 1,2\right\} $, $S_{0}\left(G\right)=S_{0}\left(G_{i}\right)$
and $G_{i}$ is itself a tree since it is obtained by omitting a leaf
from a tree. Thus $G_{i}\in\mathbb{G}$ for $i\in\left\{ 1,2,\right\} $
and $f$ is well defined. We may therefore apply $f$ repeatedly.
Since $\text{\ensuremath{\left|\mathbb{G}\right|}}<\infty$, there
exists $m_{0}>\ell_{0}\in\mathbb{N}\cup\left\{ 0\right\} $ for which
$f^{m_{0}}\left(\left(G,\mathbf{N}_{G}\right)\right)=f^{\ell_{0}}\left(\left(G,\mathbf{N}_{G}\right)\right)$.
Denote $\left(G_{\text{new}},\mathbf{N}_{G_{\text{new}}}\right):=f^{\ell_{0}}\left(\left(G,\mathbf{N}_{G}\right)\right)\in\mathbb{G}\times\Psi$,
then $f^{m_{0}-\ell_{0}}\left(\left(G_{\text{new}},\mathbf{N}_{G_{\text{new}}}\right)\right)=\left(G_{\text{new}},\mathbf{N}_{G_{\text{new}}}\right)$
which means that $G_{\text{new}}$ is a tree with a leaf as in Case
$2.1$, which we denote again by $\sigma_{\ell}$. As we have already
explained 
\begin{equation}
\mathcal{G}\left(G;\mathbf{N}_{G}\right)\leq\mathcal{G}\left(G_{\text{new}};\mathbf{N}_{G_{\text{new}}}\right),\label{eq:12.2}
\end{equation}
and $S_{0}\left(G\right)=S_{0}\left(G_{\text{new }}\right)$. Inequality
\eqref{eq:12.2} together with the induction hypothesis \eqref{eq:12}
yields the following bound 
\begin{alignat*}{1}
\mathcal{G}\left(G_{u};\mathbf{N}_{G_{u}}\right) & \leq\prod_{s=1}^{\left|S'\right|}\left[\sum_{\mathbf{v}\in[n]^{|S_{0}\left(G_{\text{new}}\right)|}}\left(\sum_{i\notin\mathbf{v}\left(\sigma_{s}\right)}b_{i\cup\mathbf{v}\left(\sigma_{s}\right)}^{\left(\mathcal{M}\left(\tau_{s}\right)\right)}\right)^{q_{s}}\prod_{\tau\in S_{d}\left(G_{\text{new}}\right)}b_{\mathbf{v}\left(\tau\right)}^{\left(\mathcal{N}_{G_{\text{new}}}\left(\tau\right)\right)}\right]^{\frac{1}{\alpha_{s}}},
\end{alignat*}
where $S'$ is the set of $d$-cells removed so far.

We are now turning to bound from above each term in the above multiplication.
Fix $1\leq s\leq\left|S'\right|$ 

\begin{alignat}{1}
 & \sum_{\mathbf{v}\in[n]^{|S_{0}\left(G_{\text{new}}\right)|}}\left(\sum_{i\notin\mathbf{v}\left(\sigma_{1}\right)}b_{i\cup\mathbf{v}\left(\sigma_{1}\right)}^{\left(\mathcal{M}\left(\tau_{s}\right)\right)}\right)^{q_{s}}\prod_{\tau\in S_{d}\left(G_{\text{new}}\right)}b_{\mathbf{v}\left(\tau\right)}^{\left(\mathcal{N}_{G_{\text{new}}}\left(\tau\right)\right)}\nonumber \\
 & =\sum_{\mathbf{v}\in[n]^{|S_{0}\left(G_{\text{new}}\right)|}}\left(\sum_{i\notin\mathbf{v}\left(\sigma_{1}\right)}b_{i\cup\mathbf{v}\left(\sigma_{1}\right)}^{\left(\mathcal{M}\left(\tau_{s}\right)\right)}\right)^{q_{s}}b_{\mathbf{v}\left(\tau_{\ell}\right)}^{\left(\mathcal{N}_{\ensuremath{G_{\text{new}}}}\left(\tau_{\ell}\right)\right)}\prod_{\tau\in S_{d}\left(G_{\text{new}}\right)\backslash\left\{ \tau_{\ell}\right\} }b_{\mathbf{v}\left(\tau\right)}^{\left(\mathcal{N}_{G_{\text{new}}}\left(\tau\right)\right)}\nonumber \\
 & \overset{\left(1\right)}{=}\sum_{\mathbf{v}\in[n]^{|S_{0}\left(G_{\text{new}}\right)|}}\left(\sum_{i\notin\mathbf{v}\left(\sigma_{i_{\ell}}\right)}b_{i\cup\mathbf{v}\left(\sigma_{i_{\ell}}\right)}^{\left(\mathcal{M}\left(\tau_{s}\right)\right)}\right)^{q_{s}}b_{\mathbf{v}\left(\tau_{\ell}\right)}^{\left(\mathcal{N}_{\ensuremath{G_{\text{new}}}}\left(\tau_{\ell}\right)\right)}\prod_{\tau\in S_{d}\left(G_{\text{new}}\right)\backslash\left\{ \tau_{\ell}\right\} }b_{\mathbf{v}\left(\tau\right)}^{\left(\mathcal{N}_{G_{\text{new}}}\left(\tau\right)\right)},\label{eq:12.4}
\end{alignat}
where $\left(1\right)$ follows since $b_{\tau}^{\left(\mathcal{M}\left(\tau_{s}\right)\right)}$
does not depend on the $d$-cell $\tau$.

We define a new graph induced from $G_{\text{new}}$, denoted as $G'$,
by omitting the $0-$cell $i_{\ell}$. Thus $S_{0}\left(G'\right)=S_{0}\left(G_{\text{new}}\right)\backslash\left\{ i_{\ell}\right\} $
and $S_{d}\left(G'\right)=S_{d}\left(G_{\text{new}}\right)\backslash\left\{ \tau_{\ell}\right\} .$
Moreover, for any $\tau\in S_{d}\left(G'\right)$, define 
\begin{equation}
\mathcal{N}_{G'}\left(\tau\right):=\mathcal{N}_{G_{\text{new}}}\left(\tau\right).\label{eq:12.7}
\end{equation}
For any $n\geq m>d$
\begin{alignat}{1}
\left|\left[n\right]^{m}\right| & =\frac{n!}{\left(n-m\right)!}=\left(n-m+1\right)\frac{n!}{\left(n-m+1\right)!}=\left(n-m+1\right)\left|\left[n\right]^{m-1}\right|\nonumber \\
 & <\left(n-d\right)\left|\left[n\right]^{m-1}\right|.\label{eq:12.6}
\end{alignat}
By the induction hypothesis $\left|S_{0}\left(G_{\text{new}}\right)\right|=\left|S_{0}\left(G\right)\right|\geq d+1$
and $\left|S_{0}\left(G'\right)\right|=\left|S_{0}\left(G_{\text{new}}\right)\right|-1$,
thus we may conclude from \eqref{eq:12.6} and \eqref{eq:12.4} that
$\mathcal{G}\left(G_{u};\mathbf{N}_{G_{u}}\right)$ is bounded from
above by
\begin{equation}
\sum_{\mathbf{v}\in[n]^{|S_{0}\left(G'\right)|}}\left(\sum_{i\notin\mathbf{v}\left(\sigma_{i_{\ell}}\right)}b_{i\cup\mathbf{v}\left(\sigma_{i_{\ell}}\right)}^{\left(\mathcal{M}\left(\tau_{s}\right)\right)}\right)^{q_{s}}\left(\sum_{i\notin\mathbf{v}\left(\sigma_{i_{\ell}}\right)}b_{i\cup\mathbf{v}\left(\sigma_{i_{\ell}}\right)}^{\left(\mathcal{N}_{\ensuremath{G_{\text{new}}}}\left(\tau_{\ell}\right)\right)}\right)\prod_{\tau\in S_{d}\left(G'\right)}b_{\mathbf{v}\left(\tau\right)}^{\left(\mathcal{N}_{G'}\left(\tau\right)\right)},\label{eq:12.5}
\end{equation}
where $\mathbf{v}\left(\sigma_{i_{\ell}}\right)$ is well defined
since $i_{\ell}\notin\sigma_{i_{\ell}}\subseteq S_{0}\left(G'\right)$.
Since,
\begin{alignat*}{1}
\eqref{eq:12.5} & =\sum_{\mathbf{v}\in[n]^{|S_{0}\left(G\right)|}}\Bigg(\Bigg(\sum_{i\notin\mathbf{v}\left(\sigma_{i_{\ell}}\right)}b_{i\cup\mathbf{v}\left(\sigma_{i_{\ell}}\right)}^{_{\left(\mathcal{M}\left(\tau_{s}\right)\right)}}\Bigg)^{q'_{s}}\!\!\!\prod_{\tau\in S_{d}\left(G'\right)}b_{\mathbf{v}\left(\tau\right)}^{_{\left(\mathcal{N}_{G'}\left(\tau\right)\right)}}\Bigg)^{\frac{q_{s}}{q'_{s}}}\Bigg(\Bigg(\sum_{i\notin\mathbf{v}\left(\sigma_{i_{\ell}}\right)}b_{i\cup\mathbf{v}\left(\sigma_{i_{\ell}}\right)}^{_{\left(\mathcal{N}_{\ensuremath{G_{\text{new}}}}\left(\tau_{\ell}\right)\right)}}\Bigg)^{q_{\ell}}\!\!\!\prod_{\tau\in S_{d}\left(G'\right)}b_{\mathbf{v}\left(\tau\right)}^{_{\left(\mathcal{N}_{G'}\left(\tau\right)\right)}}\Bigg)^{\frac{1}{q_{\ell}}},
\end{alignat*}
where $q'_{s}=\sum_{\tau\in S'\cup\left\{ \tau_{\ell}\right\} }\frac{p_{\tau_{s}}}{p_{\tau}}$,
$q_{\ell}=\sum_{\tau\in S'\cup\left\{ \tau_{\ell}\right\} }\frac{p_{\tau_{\ell}}}{p_{\tau}}$,
and one can readily verify that $\frac{q_{s}}{q'_{s}}+\frac{1}{q_{\ell}}=1$,it
follows from Hölder's inequality that \eqref{eq:12.5} is bounded
from above by
\begin{gather}
\Bigg[\sum_{_{\mathbf{v}\in[n]^{|S_{0}\left(G\right)|}}}\Bigg(\sum_{_{i\notin\mathbf{v}\left(\sigma_{i_{\ell}}\right)}}b_{i\cup\mathbf{v}\left(\sigma_{i_{\ell}}\right)}^{_{\left(\mathcal{M}\left(\tau_{s}\right)\right)}}\Bigg)^{q'_{s}}\!\!\!\prod_{_{\tau\in S_{d}\left(G'\right)}}\!\!b_{\mathbf{v}\left(\tau\right)}^{_{\left(\mathcal{N}_{G'}\left(\tau\right)\right)}}\Bigg]^{\frac{q_{s}}{q'_{s}}}\Bigg[\Bigg(\sum_{_{\mathbf{v}\in[n]^{|S_{0}\left(G\right)|}}}\Bigg(\sum_{_{i\notin\mathbf{v}\left(\sigma_{i_{\ell}}\right)}}b_{i\cup\mathbf{v}\left(\sigma_{i_{\ell}}\right)}^{_{\left(\mathcal{N}_{\ensuremath{G_{\text{new}}}}\left(\tau_{\ell}\right)\right)}}\Bigg)^{q_{\ell}}\!\!\!\prod_{_{\tau\in S_{d}\left(G'\right)}}\!\!b_{\mathbf{v}\left(\tau\right)}^{_{\left(\mathcal{N}_{G'}\left(\tau\right)\right)}}\Bigg)\Bigg]^{\frac{1}{q_{\ell}}},\label{eq:12.3}
\end{gather}
with the same degree of homogeneity in each variable $\left\{ b^{\left(\mathcal{N}\left(\tau\right)\right)}\right\} $
as of those in term \eqref{eq:12.5}.

Denoting $\tau_{r+1}:=\tau_{\ell}$ and $S=S'\cup\left\{ \tau_{r+1}\right\} $,
and replacing every term in the induction hypothesis \eqref{eq:12}
by the upper bound in \eqref{eq:12.3}, gives 
\begin{equation}
\mathcal{G}\left(G_{u};\mathbf{N}_{G_{u}}\right)\leq\prod_{s=1}^{\left|S\right|}\left[\sum_{\mathbf{v}\in[n]^{|S_{0}\left(G'\right)|}}\left(\sum_{i\notin\mathbf{v}\left(\sigma_{i_{\ell}}\right)}b_{i\cup\mathbf{v}\left(\sigma_{i_{\ell}}\right)}^{\left(\mathcal{M}\left(\tau_{s}\right)\right)}\right)^{\widehat{q_{s}}}\prod_{\tau\in S_{d}\left(G'\right)}b_{\mathbf{v}\left(\tau\right)}^{\left(\mathcal{N}_{G'}\left(\tau\right)\right)}\right]^{\frac{1}{\widehat{\alpha_{s}}}},\label{eq:13}
\end{equation}
where 
\[
\widehat{q_{s}}:=\begin{cases}
q'_{s} & \text{for \ensuremath{1\leq s\leq r}}\\
q_{\ell} & \text{for \ensuremath{s=r+1}}
\end{cases},\qquad\mathcal{M}\left(\tau_{s}\right):=\begin{cases}
\mathcal{M}\left(\tau_{s}\right) & \text{for \ensuremath{1\leq s\leq r}}\\
\mathcal{N}_{\ensuremath{G_{\text{new}}}}\left(\tau_{\ell}\right) & \text{for \ensuremath{s=r+1}}
\end{cases},\qquad\widehat{\alpha_{s}}:=\begin{cases}
\frac{\alpha_{s}q'_{s}}{q_{s}} & \text{for \ensuremath{1\leq s\leq r}}\\
\frac{q_{\ell}}{\sum_{j=1}^{r}\frac{1}{\alpha_{s}}} & \text{for \ensuremath{s=r+1}}
\end{cases}.
\]

Following the above construction, it is clear that all conditions
besides 2(d) are satisfied by the relative parameters in \eqref{eq:13},
where for 2(d) we observe that by the induction hypothesis
\begin{alignat*}{1}
\sum_{\tau\in S_{d}\left(G_{\mathbf{u}}\right)}\mathcal{N}_{G_{\mathbf{u}}}\left(\tau\right) & =\sum_{s=1}^{r}\mathcal{M}\left(\tau_{s}\right)+\sum_{\tau\in S_{d}\left(G\right)}\mathcal{N}_{G}\left(\tau\right)\\
 & \overset{\left(1\right)}{=}\sum_{s=1}^{r}\mathcal{M}\left(\tau_{s}\right)+\sum_{\tau\in S_{d}\left(G_{\text{new}}\right)}\mathcal{N}_{G_{\text{new}}}\left(\tau\right)\\
 & =\sum_{s=1}^{r}\mathcal{M}\left(\tau_{s}\right)+\mathcal{N}_{G_{\text{new}}}\left(\tau_{\ell}\right)+\sum_{\tau\in S_{d}\left(G_{\text{new}}\right)\backslash\left\{ \tau_{\ell}\right\} }\mathcal{N}_{G_{\text{new}}}\left(\tau\right)\\
 & \overset{\left(2\right)}{=}\sum_{s=1}^{r+1}\mathcal{M}\left(\tau_{s}\right)+\sum_{\tau\in S_{d}\left(G'\right)}\mathcal{N}_{G'}\left(\tau\right),
\end{alignat*}
where $\left(1\right)$ follows by the construction of $G_{\text{new}}$
using the function $f$ and $\left(2\right)$ follows from our definition
in \eqref{eq:12.7}. This proves the induction step, $r\to r+1$.

The above induction guarantees the validity of the induction hypothesis
for $r=\left|S_{0}\left(G_{u}\right)\right|-d$ , that is, we have
proved the following bound

\begin{gather*}
\mathcal{G}\left(G_{u};\mathbf{N}_{G_{u}}\right)\leq\prod_{s=1}^{\left|S\right|}\left[\sum_{\mathbf{v}\in[n]^{d}}\left(\sum_{i\notin\mathbf{v}\left(\sigma_{0}\right)}b_{i\cup\mathbf{v}\left(\sigma_{0}\right)}^{\left(\mathcal{M}\left(\tau_{s}\right)\right)}\right)^{\sum_{\tau\in S_{G_{u}}}\frac{p_{\tau_{s}}}{p_{\tau}}}\right]^{\frac{1}{\alpha_{s}}},
\end{gather*}
where $\sigma_{0}=\left[1,2,\ldots,d\right]\in X_{\pm}^{d-1}$ and
$S\subseteq S_{d}\left(G_{u}\right)$ with $\left|S\right|=\left|S_{0}\left(G_{u}\right)\right|-d$
is the set of $d$-cells removed in the process. Recall that we assume
$\sum_{\tau\in S}\frac{1}{p_{\tau}}=1$, hence we conclude that $\mathcal{G}\left(G_{u};\mathbf{N}_{G_{u}}\right)$
is bounded from above by 
\begin{equation}
\prod_{s=1}^{\left|S\right|}\left[\sum_{\mathbf{v}\in[n]^{d}}\left(\sum_{i\notin\mathbf{v}\left(\sigma_{0}\right)}b_{i\cup\mathbf{v}\left(\sigma_{0}\right)}^{\left(\mathcal{M}\left(\tau_{s}\right)\right)}\right)^{p_{\tau_{s}}}\right]^{\frac{1}{\alpha_{s}}}.\label{eq:4}
\end{equation}
The induction argument guarantees that the term in equation \eqref{eq:4}
is $1$-homogeneous in all variables $b^{\left(\mathcal{M}\left(\tau_{s}\right)\right)}$.
It must therefore necessarily be the case that $\alpha_{s}=p_{\tau_{s}}$.
Consequently
\begin{alignat*}{1}
\mathcal{G}\left(G_{u};\mathbf{N}_{G_{u}}\right) & \leq\prod_{s=1}^{\left|S\right|}\left[\sum_{\mathbf{v}\in[n]^{d}}\left(\sum_{i\notin\mathbf{v}\left(\sigma_{0}\right)}b_{i\cup\mathbf{v}\left(\sigma_{0}\right)}^{\left(\mathcal{M}\left(\tau_{s}\right)\right)}\right)^{p_{\tau_{s}}}\right]^{\frac{1}{p_{\tau_{s}}}}\\
 & \overset{\left(1\right)}{=}\prod_{\tau\in S}\left[d!\sum_{\sigma\in X_{+}^{d-1}}\left(\sum_{i\in\left[n\right],i\notin\sigma}b_{\sigma\omega^{\left(\sigma,i\right)}}^{\left(\mathcal{M}\left(\tau\right)\right)}\right)^{p_{\tau}}\right]^{\frac{1}{p_{\tau}}}\\
 & \overset{\left(2\right)}{=}d!\prod_{\tau\in S}\left[\sum_{\sigma\in X_{+}^{d-1}}\left(\sum_{i\in\left[n\right],i\notin\sigma}b_{\sigma\omega^{\left(\sigma,i\right)}}^{\left(\mathcal{M}\left(\tau\right)\right)}\right)^{p_{\tau}}\right]^{\frac{1}{p_{\tau}}},
\end{alignat*}
where $\omega^{\left(\sigma,i\right)}\in\Sigma_{\sigma,i}$ , $\left(1\right)$
follows from the same arguments which were stated in the initial step
and $\left(2\right)$ follows because $\sum_{\tau\in S}\frac{1}{p_{\tau}}=1$.
This concludes the proof of Lemma \ref{Lem:lemma_2}.
\end{itemize}
\end{proof}

\subsubsection{Back to the proof of Theorem \ref{thm:main_result}}

Let $w\in\mathcal{W}_{2k+1}$ and $\mathbf{N}_{G_{w}}=\left(N_{w}\left(e\right)\right)_{e\in E_{w}}$.
This choice of labeling implies that $\mathcal{N}_{G_{w}}\left(\tau\right)=N_{w}\left(\tau\right)$
for all $\tau\in\text{supp}_{d}\left(w\right).$ Let $T$ be the graph
obtained from Lemma \ref{lem:Lemma1}, and apply Lemma \ref{Lem:lemma_2}
to it and its labeling $\left(\mathcal{N}_{T}\left(\tau\right)\right)_{\tau\in S_{d}\left(T\right)}$,
using $p_{\tau}=\frac{2k}{\mathcal{M}\left(\tau\right)}$ for $\tau\in S$,
where $\left(\mathcal{M}\left(\tau\right)\right)_{\tau\in S}$ is
the vector from Lemma \ref{Lem:lemma_2}. Note that this is a valid
choice for $\left(p_{\tau}\right)_{\tau\in S}$ since: 
\begin{alignat*}{1}
\sum_{\tau\in S_{T}}\frac{1}{p_{\tau}} & =\frac{\sum_{\tau\in S}\mathcal{M}\left(\tau\right)}{2k}\overset{_{(1)}}{=}\frac{\sum_{\tau\in S_{d}\left(T\right)}\mathcal{N}_{T}\left(\tau\right)}{2k}\overset{_{(2)}}{=}\frac{\sum_{\tau\in S_{d}\left(G_{w}\right)}\mathcal{N}_{G_{w}}\left(\tau\right)}{2k}=\frac{\sum_{\tau\in S_{d}\left(G_{w}\right)}N_{w}\left(\tau\right)}{2k}\overset{_{(3)}}{=}\frac{2k}{2k}=1,
\end{alignat*}
where $(1)$ follows from Lemma \ref{Lem:lemma_2}, $(2)$ follows
from Lemma \ref{lem:Lemma1}, $(3)$ follows since $w\in\mathcal{W}_{2k+1}$.

Hence, by Lemma \ref{Lem:lemma_2}
\begin{equation}
\mathcal{G}\left(T;\mathbf{N}_{T}\right)\leq d!\cdot\prod_{\tau\in S}\left[\sum_{\sigma\in X_{+}^{d-1}}\left(\sum_{i\in\left[n\right],i\notin\sigma}b_{\sigma\omega^{\left(\sigma,i\right)}}^{\left(\mathcal{M}\left(\tau\right)\right)}\right)^{\frac{2k}{\mathcal{M}\left(\tau\right)}}\right]^{\frac{\mathcal{M}\left(\tau\right)}{2k}}.\label{eq:50}
\end{equation}
Recall that for $k\in\mathbb{N}$, we defined
\[
\theta_{k}:=\sqrt{\frac{n-d}{n}}{n \choose d}^{\frac{1}{2k}}\qquad,\qquad\theta_{k}^{*}:=\left\Vert Z_{\tau}-\mathbb{E}\left[Z\right]\right\Vert _{\infty}\left({n \choose d}\cdot\frac{d\left(n-d\right)}{\left(n\text{Var}\left(Z\right)\right)^{k}}\right)^{\frac{1}{2k}},
\]
which one can verify equal to 
\[
\theta_{k}:=\left(\sum_{\sigma\in X_{+}^{d-1}}\left(\sum_{i\in\left[n\right],i\notin\sigma}\mathbb{E}\left[\left|H_{\sigma\omega^{\left(\sigma,i\right)}}\right|^{2}\right]\right)^{k}\right)^{\frac{1}{2k}}\qquad\text{and}\qquad\text{ }\theta_{k}^{*}:=\left(\sum_{\sigma,\sigma'\in X_{+}^{d-1}}\left\Vert H_{\sigma\sigma'}\right\Vert _{\infty}^{2k}\right)^{\frac{1}{2k}},
\]
where $\omega^{\left(i\right)}\in\Sigma_{\sigma,i}$ is arbitrary.
Observe the product term on \eqref{eq:50} and note that
\begin{alignat}{1}
\sum_{\sigma\in X_{+}^{d-1}}\left(\sum_{i\in\left[n\right],i\notin\sigma}b_{\sigma\omega^{\left(\sigma,i\right)}}^{\left(\mathcal{M}\left(\tau\right)\right)}\right)^{\frac{2k}{\mathcal{M}\left(\tau\right)}} & =\sum_{\sigma\in X_{+}^{d-1}}\left(\sum_{i\in\left[n\right],i\notin\sigma}\mathbb{E}\left[\left|B_{\sigma\omega^{\left(\sigma,i\right)}}\right|^{\mathcal{M}\left(\tau\right)}\right]\right)^{\frac{2k}{\mathcal{M}\left(\tau\right)}}.\label{eq:41}
\end{alignat}
Since Lemma \ref{Lem:lemma_2}, Lemma \ref{lem:Lemma1} and the choice
of $\mathbf{N}_{G_{w}}$ together with \eqref{eq:0} implies 
\begin{gather*}
\mathcal{M}\left(\tau\right)\geq\mathcal{N}_{T}\left(\tau\right)\geq\mathcal{N}_{G_{w}}\left(\tau\right)=N_{w}\left(\tau\right)\geq2,
\end{gather*}
it follows that
\begin{alignat*}{1}
\text{LHS of \eqref{eq:41}} & =\sum_{\sigma\in X_{+}^{d-1}}\left(\sum_{i\in\left[n\right],i\notin\sigma}\mathbb{E}\left[\left|B_{\sigma\omega^{\left(\sigma,i\right)}}\right|^{2}\left|B_{\sigma\omega^{\left(\sigma,i\right)}}\right|^{\mathcal{M}\left(\tau\right)-2}\right]\right)^{\frac{2k}{\mathcal{M}\left(\tau\right)}}\\
 & \leq\sum_{\sigma\in X_{+}^{d-1}}\left[\left(\sum_{i\in\left[n\right],i\notin\sigma}\mathbb{E}\left[\left|B_{\sigma\omega^{\left(\sigma,i\right)}}\right|^{2}\right]\right)^{\frac{2k}{\mathcal{M}\left(\tau\right)}}\cdot\max_{\sigma'\in X_{+}^{d-1}}\left\Vert B_{\sigma\sigma'}\right\Vert _{\infty}^{\frac{2k\left(\mathcal{M}\left(\tau\right)-2\right)}{\mathcal{M}\left(\tau\right)}}\right].
\end{alignat*}
Applying Hölder's inequality with $\frac{2}{\mathcal{M}\left(\tau\right)}+\frac{\mathcal{M}\left(\tau\right)-2}{\mathcal{M}\left(\tau\right)}=1$,
the last expression is bounded from above by 
\begin{alignat*}{1}
 & \left(\sum_{\sigma\in X_{+}^{d-1}}\left(\sum_{i\in\left[n\right],i\notin\sigma}\mathbb{E}\left[\left|B_{\sigma\omega^{\left(\sigma,i\right)}}\right|^{2}\right]\right)^{k}\right)^{\frac{2}{\mathcal{M}\left(\tau\right)}}\cdot\left(\sum_{\sigma\in X_{+}^{d-1}}\max_{\sigma'\in X_{+}^{d-1}}\left\Vert B_{\sigma\sigma'}\right\Vert _{\infty}^{2k}\right)^{\frac{\mathcal{M}\left(\tau\right)-2}{\mathcal{M}\left(\tau\right)}}\\
 & \leq\left(\sum_{\sigma\in X_{+}^{d-1}}\left(\sum_{i\in\left[n\right],i\notin\sigma}\mathbb{E}\left[\left|B_{\sigma\omega^{\left(\sigma,i\right)}}\right|^{2}\right]\right)^{k}\right)^{\frac{2}{\mathcal{M}\left(\tau\right)}}\cdot\left(\sum_{\sigma,\sigma'\in X_{+}^{d-1}}\left\Vert B_{\sigma\sigma'}\right\Vert _{\infty}^{2k}\right)^{\frac{\mathcal{M}\left(\tau\right)-2}{\mathcal{M}\left(\tau\right)}}\\
 & =\left(n\text{Var}\left(Z\right)\right)^{k}\theta_{k}^{\frac{4k}{\mathcal{M}\left(\tau\right)}}\cdot\left(\theta_{k}^{*}\right)^{\frac{2k\left(\mathcal{M}\left(\tau\right)-2\right)}{\mathcal{M}\left(\tau\right)}}.
\end{alignat*}
Combining all of the above gives
\begin{alignat}{1}
\mathcal{G}\left(T;\mathbf{N}_{T}\right) & \leq d!\cdot\prod_{\tau\in S}\left[\left(n\text{Var}\left(Z\right)\right)^{k}\theta_{k}^{\frac{4k}{\mathcal{M}\left(\tau\right)}}\cdot\left(\theta_{k}^{*}\right)^{\frac{2k\left(\mathcal{M}\left(\tau\right)-2\right)}{\mathcal{M}\left(\tau\right)}}\right]^{\frac{\mathcal{M}\left(\tau\right)}{2k}}\nonumber \\
 & =d!\cdot\left(n\text{Var}\left(Z\right)\right)^{\frac{\sum_{\tau\in S}\mathcal{M}\left(\tau\right)}{2}}\theta_{k}^{2\left|S\right|}\cdot\left(\theta_{k}^{*}\right)^{\left(\sum_{\tau\in S}\mathcal{M}\left(\tau\right)\right)-2\left|S\right|}.\label{eq:51}
\end{alignat}
We turn to estimate the powers in \eqref{eq:51}. Using once more
Lemma \ref{Lem:lemma_2}, Lemma \ref{lem:Lemma1}, the choice of $\mathbf{N}_{G_{w}}$
and that $w\in\mathcal{W}_{2k+1}$ gives 
\begin{alignat*}{1}
\sum_{\tau\in S}\mathcal{M}\left(\tau\right) & =\sum_{\tau\in S_{d}\left(T\right)}\mathcal{N}_{T}\left(\tau\right)=\sum_{\tau\in S_{d}\left(G_{w}\right)}\mathcal{N}_{G_{w}}\left(\tau\right)=\sum_{\tau\in\text{supp}_{d}\left(w\right)}N_{w}\left(\tau\right)=2k
\end{alignat*}
and
\[
\left|S\right|=\left|S_{0}\left(T\right)\right|-d=\left|S_{0}\left(G_{w}\right)\right|-d.
\]
Hence, by Lemma \ref{lem:Lemma1}
\[
\mathcal{G}\left(G_{w};\mathbf{N}_{G_{w}}\right)\leq\mathcal{G}\left(T;\mathbf{N}_{T}\right)\leq d!\cdot\left(n\text{Var}\left(Z\right)\right)^{k}\cdot\theta_{k}^{2\left(\left|S_{0}\left(G_{w}\right)\right|-d\right)}\cdot\left(\theta_{k}^{*}\right)^{2k-2\left(\left|S_{0}\left(G_{w}\right)\right|-d\right)}.
\]
Using \eqref{eq:2}, we conclude
\begin{alignat*}{1}
\mathbb{E}\left[\left\Vert H\right\Vert _{S_{2k}}^{2k}\right] & \leq d!\cdot\sum_{w\in\mathcal{W}_{2k+1}}\theta_{k}^{2\left(\left|S_{0}\left(G_{w}\right)\right|-d\right)}\cdot\left(\theta_{k}^{*}\right)^{2k-2\left(\left|S_{0}\left(G_{w}\right)\right|-d\right)}.
\end{alignat*}
By rescaling the matrix $H$, we may assume without loss of generality
that $\theta_{k}^{*}=1$. Consequently,
\begin{alignat}{1}
\mathbb{E}\left[\left\Vert H\right\Vert _{S_{2k}}^{2k}\right] & \leq d!\sum_{w\in\mathcal{W}_{2k+1}}\theta_{k}^{2\left(\left|S_{0}\left(G_{w}\right)\right|-d\right)}=d!\sum_{w\in\mathcal{W}_{2k+1}}\theta_{k}^{2\left(\left|\text{supp}_{0}\left(w\right)\right|-d\right)}\label{eq:5}
\end{alignat}

We will now show that the right hand side of \eqref{eq:5} is bounded
from above by a function depending on $\mathbb{E}\left[\left\Vert Y\right\Vert _{S_{2k}}^{2k}\right]$,
where $Y$ is the matrix from Section \ref{sec:Norm-bounds-for} with
$p_{0}=\frac{1}{4}$. As we have already shown
\begin{alignat}{1}
\mathbb{E}\left[\left\Vert Y\right\Vert _{S_{2k}}^{2k}\right] & =\sum_{w\in\mathcal{W}_{2k+1}}\sum_{u\sim w}\prod_{\tau\in\text{supp}_{d}\left(u\right)}\mathbb{E}\left[Y_{\tau}^{N_{u}\left(\tau\right)}\right]\nonumber \\
 & \overset{\left(1\right)}{\geq}\sum_{w\in\mathcal{W}_{2k+1}}\sum_{u\sim w}1\nonumber \\
 & =\sum_{w\in\mathcal{W}_{2k+1}}\#\left\{ u\text{ ; \ensuremath{u} is a word with \ensuremath{u\sim w}}\right\} ,\label{eq:45}
\end{alignat}
where $\left(1\right)$ follows since $\mathbb{E}\left[Y_{\tau}^{m}\right]\geq1$
for all $m\geq2$, and $\mathbb{E}\left[Y_{\tau}^{m}\right]=0$ for
$m=1$. We turn to estimate the sum in \eqref{eq:45} showing that
\[
\#\left\{ u\text{ a word ; \ensuremath{u\sim w}}\right\} \geq\frac{\left(r-d+1\right)!}{\left(r-\left|\text{supp}_{0}\left(w\right)\right|\right)!}.
\]
Indeed, fix $w\in\mathcal{W}_{2k+1}$, with $w=\sigma_{1}\cdots\sigma_{2k}\sigma_{1}$,
and $\sigma_{1}=\left[\sigma_{1}^{0},\ldots,\sigma_{1}^{d-1}\right]$.
Define a set of permutations on $\left[r\right]$, denoted $\Psi$,
for which each permutation $\pi$ fixes $\left(\sigma_{1}^{i}\right)_{i=1}^{d-1}$,
takes $\sigma_{1}^{0}$ to some number in the set $\left[r\right]\backslash\left\{ \sigma_{1}^{j}\right\} _{j=1}^{d-1}$,
and each new appearance of $0-$cell, takes to some number in $\left[r\right]$
which did not appear earlier. Each permutation in the above set, induces
a new word which is equivalent to $w$. Observe that each choice on
the image of $\sigma_{1}^{0}$ defines a different word, since the
image of $\sigma_{1}$ would be different (because $\left(\sigma_{1}^{i}\right)_{i=1}^{d-1}$
are fixed). Consequently, 
\[
\#\left\{ u\text{ a word ; \ensuremath{u\sim w}}\right\} \geq\left|\Psi\right|=\left(r-\left(d-1\right)\right)\left(r-d\right)\cdots\left(r-\left(\text{supp}_{0}\left(w\right)-1\right)\right)=\frac{\left(r-d+1\right)!}{\left(r-\left|\text{supp}_{0}\left(w\right)\right|\right)!}.
\]
Using the above estimation in \eqref{eq:45} gives
\begin{equation}
\mathbb{E}\left[\left\Vert Y\right\Vert _{S_{2k}}^{2k}\right]\geq\sum_{w\in\mathcal{W}_{2k+1}}\frac{\left(r-d+1\right)!}{\left(r-\left|\text{supp}_{0}\left(w\right)\right|\right)!}.\label{eq:46}
\end{equation}
From \eqref{eq:0} we know that $\left|\text{supp}_{d}\left(w\right)\right|<k+1$
and since $\left|\text{supp}_{0}\left(w\right)\right|\le\left|\text{supp}_{d}\left(w\right)\right|+d$,
we obtain $\left|\text{supp}_{0}\left(w\right)\right|<k+d+1.$ Set
$r=\left[\theta_{k}^{2}\right]+k+d+1$. For $\ell_{0}:=r-d+1$ and
$m_{0}:=\left|\text{supp}_{0}\left(w\right)\right|-d+1$, we observe
that $\left|\text{supp}_{0}\left(w\right)\right|<k+d+1\leq r$ implies
$m_{0}<\ell_{0}$. As $\frac{\left(\ell-1\right)!}{\left(\ell-m\right)!}\geq\left(\ell-m+1\right)^{m-1}$
for any $\ell\geq m$ and thus
\begin{alignat*}{1}
\frac{\left(r-d\right)!}{\left(r-\left|\text{supp}_{0}\left(w\right)\right|\right)!} & \geq\left(r-\left|\text{supp}_{0}\left(w\right)\right|+1\right)^{\left|\text{supp}_{0}\left(w\right)\right|-d}\\
 & =\left(\left[\theta_{k}^{2}\right]+k+d+1-\left|\text{supp}_{0}\left(w\right)\right|+1\right)^{\left|\text{supp}_{0}\left(w\right)\right|-d}\\
 & \geq\left(\left[\theta_{k}^{2}\right]+1\right)^{\left|\text{supp}_{0}\left(w\right)\right|-d}\\
 & \geq\theta_{k}^{2\left(\left|\text{supp}_{0}\left(w\right)\right|-d\right)}.
\end{alignat*}
The last bound, when applied to \eqref{eq:46} yields
\[
\mathbb{E}\left[\left\Vert Y\right\Vert _{S_{2k}}^{2k}\right]\geq\left(r-d+1\right)\sum_{w\in\mathcal{W}_{2k+1}}\theta_{k}^{2\left(\left|\text{supp}_{0}\left(w\right)\right|-d\right)},
\]
which then by \eqref{eq:5} gives 
\[
\frac{1}{d!}\mathbb{E}\left[\left\Vert H\right\Vert _{S_{2k}}^{2k}\right]\leq\frac{1}{\left(r-d+1\right)}\mathbb{E}\left[\left\Vert Y\right\Vert _{S_{2k}}^{2k}\right],
\]
and hence
\[
\mathbb{E}\left[\left\Vert H\right\Vert _{S_{2k}}^{2k}\right]\leq\frac{d!}{r-d+1}\mathbb{E}\left[\left\Vert Y\right\Vert _{S_{2k}}^{2k}\right]\leq\frac{{r \choose d}d!}{r-d+1}\mathbb{E}\left[\left\Vert Y\right\Vert _{2}^{2k}\right].
\]
 Applying Proposition \eqref{prop:first_Bound_on_Y} and Proposition
\eqref{propr:Second_estimation_on_Y} to the matrix $Y$, we conclude
that 
\begin{alignat*}{1}
\sqrt[2k]{\mathbb{E}\left[\left\Vert H\right\Vert _{S_{2k}}^{2k}\right]} & \leq\sqrt[2k]{\frac{d!}{r-d+1}{r \choose d}}\left(2\sqrt{dr}+C_{d}r^{1/3}\log^{2/3}r+c_{d}\sqrt{k}\right)\\
 & \leq\sqrt[2k]{\frac{d!}{1-\frac{d-1}{r}}}\left(2\sqrt{d}\left(\sqrt{r}\right)^{1+\frac{d-1}{k}}+C_{d}\left(\sqrt{r}\right)^{\frac{2}{3}+\frac{d-1}{k}}\log^{2/3}r+c_{d}\left(\sqrt{r}\right)^{\frac{d-1}{k}}\sqrt{k}\right),
\end{alignat*}
where $C_{d}$ and $c_{d}$ are positive constants depending only
on $d$.

Because $r=\left[\theta_{k}^{2}\right]+k+d+1$, it follows that $\sqrt{r}\leq\theta_{k}+\sqrt{k+d}$
, which together with the assumption $k\geq d$ and the choice of
normalization $\theta_{k}^{*}=1$ gives 
\begin{equation}
\sqrt[2k]{\mathbb{E}\left[\left\Vert H\right\Vert _{S_{2k}}^{2k}\right]}\leq\Phi\left(\theta_{k},\theta_{k}^{*}\right).\label{eq:4.25}
\end{equation}
This concludes the proof of Theorem \eqref{thm:main_result}.\hfill\qed

\section{The asymptotic behavior of the norm of $H$}

\subsection{Proof of Corollary \ref{cor:Norm_of_H}}

For every $C>0$ and $k=k\left(n\right)\geq C\log\left(n\right)$,
one can readily verify that for all $n$ large enough (depending only
on $C$ and $d$) $\theta_{k}\leq\sqrt{k}$. Thus, using \eqref{eq:4.25}
we obtain that for any $C>0$, all large enough $n$ and any integer
$k=k\left(n\right)$ such that $k\geq C\log\left(n\right)$ 
\begin{equation}
\sqrt[2k]{\mathbb{E}\left[\left\Vert H\right\Vert _{S_{2k}}^{2k}\right]}\leq\sqrt[2k]{d!d}\cdot\left(2\theta_{k}^{*}\sqrt{d}\left(\frac{\theta_{k}}{\theta_{k}^{*}}+2\sqrt{k}\right)^{1+\frac{d-1}{k}}+C_{d}\theta_{k}^{*}\left(\sqrt{k}\right)^{1+\frac{d-1}{k}}\right).\label{eq:4.24}
\end{equation}
We now turn to show that for an appropriate choice of $k:=k\left(n\right)$
growing to infinity with $n$
\[
\limsup_{n\to\infty}\theta_{k}^{*}\left(\frac{\theta_{k}}{\theta_{k}^{*}}+2\sqrt{k}\right)^{1+\frac{d-1}{k}}\leq1\qquad\text{and}\qquad\lim_{n\to\infty}\theta_{k}^{*}\left(\sqrt{k}\right)^{1+\frac{d-1}{k}}=0,
\]
thus proving that
\begin{equation}
\limsup_{n\to\infty}\sqrt[2k]{\mathbb{E}\left[\left\Vert H\right\Vert _{S_{2k}}^{2k}\right]}\leq2\sqrt{d.}\label{eq:24}
\end{equation}

From the definition of $\theta_{k}$ and $\theta_{k}^{*}$ 
\[
\frac{\theta_{k}}{\theta_{k}^{*}}=\sqrt{\frac{n-d}{n}}\cdot\frac{1}{\left\Vert Z_{\tau}-\mathbb{E}\left[Z\right]\right\Vert _{\infty}}\left(\frac{\left(n\text{Var}\left(Z\right)\right)^{k}}{d\left(n-d\right)}\right)^{\frac{1}{2k}}.
\]
As for the first limit, observe that 
\begin{equation}
\theta_{k}^{*}\left(\frac{\theta_{k}}{\theta_{k}^{*}}+2\sqrt{k}\right)^{1+\frac{d-1}{k}}\leq n^{\frac{d}{2k}}\left(1-\frac{d}{n}\right)^{\frac{1}{2k}}\left(\left(1-\frac{d}{n}\right)^{\frac{1}{2}-\frac{1}{2k}}+2\left(dn\right)^{\frac{1}{2k}}\left\Vert Z_{\tau}-\mathbb{E}\left[Z\right]\right\Vert _{\infty}\sqrt{\frac{k}{n\text{Var}\left(z\right)}}\right)^{1+\frac{d-1}{k}}.\label{eq:4.26}
\end{equation}
As for the second limit, note that 
\begin{equation}
\theta_{k}^{*}\left(\sqrt{k}\right)^{1+\frac{d-1}{k}}\leq n^{\frac{d+1}{2k}}\left\Vert Z_{\tau}-\mathbb{E}\left[Z\right]\right\Vert _{\infty}\sqrt{\frac{k}{n\text{Var}\left(Z\right)}}\left(d\left(1-\frac{d}{n}\right)\right)^{\frac{1}{2k}}k^{\frac{d-1}{2k}}\label{eq:4.27}
\end{equation}

Choose $k_{0}:=k_{0}\left(n\right)=\left\lceil \sqrt{n\text{Var}\left(Z\right)\log\left(n\right)}\right\rceil $.
Observe that $k_{0}\geq C\log\left(n\right)$ and thus \eqref{eq:4.24}
holds for all large enough $n$. Under the restriction $n\text{Var}\left(Z\right)\gg\log\left(n\right)$,
taking $n\to\infty$ we derive from \eqref{eq:4.26} $\limsup_{n\to\infty}\theta_{k}^{*}\left(\frac{\theta_{k}}{\theta_{k}^{*}}+2\sqrt{k}\right)^{1+\frac{d-1}{k}}\leq1,$
and from \eqref{eq:4.27} $\lim_{n\to\infty}\theta_{k}^{*}\left(\sqrt{k}\right)^{1+\frac{d-1}{k}}=0.$

In order to complete the proof, we further note that $\left\Vert H\right\Vert _{2}\leq\left\Vert H\right\Vert _{S_{2k}}$
and thus by Jensen's inequality
\[
\mathbb{E}\left[\left\Vert H\right\Vert _{2}\right]\leq\mathbb{E}\left[\left\Vert H\right\Vert _{2}^{2k}\right]^{\frac{1}{2k}}\leq\mathbb{E}\left[\left\Vert H\right\Vert _{S_{2k}}^{2k}\right]^{\frac{1}{2k}}.
\]

By taking $k_{0}$ as above we conclude that 
\begin{equation}
\limsup_{n\to\infty}\mathbb{E}\left[\left\Vert H\right\Vert _{2}\right]\leq\limsup_{n\to\infty}\sqrt[2k_{0}]{\mathbb{E}\left[\left\Vert H\right\Vert _{S_{2k_{0}}}^{2k_{0}}\right]}\underset{\text{by \eqref{eq:24}}}{\underbrace{\leq}}2\sqrt{d.}\label{eq:9}
\end{equation}
Note that under the assumption $n\text{Var}\left(Z\right)\gg\log\left(n\right)$
we have $n\text{Var}\left(Z\right)\underset{n\to\infty}{\longrightarrow}\infty$,
hence by \cite[Remark 5.2]{RK17}\footnote{This remark relates to \cite[Theorem 3.1]{RK17}, which is stated
for the matrix $\mathcal{A}$ and not $H$. However, the generalization
of Theorem $3.1$ for the matrix $H$ follows readily from the proof
presented in \cite{RK17}, and we will not present it here.} we have $\liminf_{n\to\infty}\left\Vert H\right\Vert _{2}\geq2\sqrt{d}$,
$\mathbb{P}-$almost surely. Using Fatou's lemma we deduce
\begin{equation}
\liminf_{n\to\infty}\mathbb{E}\left[\left\Vert H\right\Vert _{2}\right]\geq2\sqrt{d}.\label{eq:10}
\end{equation}
By \eqref{eq:9} and \eqref{eq:10} we obtain that under the assumption
$n\text{Var}\left(Z\right)\gg\log\left(n\right)$, we have
\[
\lim_{n\to\infty}\mathbb{E}\left[\left\Vert H\right\Vert _{2}\right]=2\sqrt{d},
\]
which concludes the proof of Corollary \ref{cor:Norm_of_H}.\hfill\qed

\subsection{Proof of Corollary \ref{cor:Norm_of_H_2}}

For a fixed $C\in(0,\infty)$ and for an integer $k\ge C\log\left(n\right)$,
we define the function $g_{k,n,Z}:[0,1]^{|K^{d}|}\to\mathbb{R}$ via
\[
g_{k,n,Z}((x_{\tau})_{\tau\in K^{d}})=\frac{\sqrt{n\text{Var}\left(Z\right)}}{e^{\frac{d}{2C}}\left(d+1\right)}\|B((x_{\tau})_{\tau\in K^{d}})\|_{S_{2k}},
\]
 where $B((x_{\tau})_{\tau\in K^{d}})$ is a $|K_{+}^{d-1}|\times|K_{+}^{d-1}$|
matrix defined by
\[
B((x_{\tau})_{\tau\in K^{d}})_{\sigma,\sigma'\in K_{+}^{d-1}}\equiv\begin{cases}
\frac{x_{\tau}-\mathbb{E}\left[Z\right]}{\sqrt{n\text{Var\ensuremath{\left(Z\right)}}}} & \text{if \ensuremath{\sigma\overset{K}{\sim}\sigma'} and \ensuremath{\sigma\cup\sigma'=\tau} }.\\
-\frac{x_{\tau}-\mathbb{E}\left[Z\right]}{\sqrt{n\text{Var\ensuremath{\left(Z\right)}}}} & \text{if \ensuremath{\sigma\overset{K}{\sim}\overline{\sigma'}} and \ensuremath{\sigma\cup\sigma'=\tau} }\\
0 & \text{otherwise}
\end{cases}
\]

Note that $H=B((Z_{\tau})_{\tau\in K^{d}})$. The inequality follows
from Talagrand\textquoteright s concentration inequality, c.f. \cite[Theorem 6.10]{BLM13},
using the fact that the function $g_{k,n,Z}$ is convex and 1-Lipschitz,
and thus for any $t>0$ 
\begin{alignat}{1}
 & \mathbb{P}\left(\left\Vert H\right\Vert _{S_{2k}}\geq\mathbb{E}\left[\left\Vert H\right\Vert _{S_{2k}}\right]+t\right)\nonumber \\
= & \mathbb{P}\left(g_{k,n,Z}\left(\left(Z_{\tau}\right)_{\tau\in X^{d}}\right)\geq\mathbb{E}\left[g_{k,n,Z}\left(\left(Z_{\tau}\right)_{\tau\in X^{d}}\right)\right]+\frac{\sqrt{n\text{Var}\left(Z\right)}}{e^{\frac{d}{2C}}\left(d+1\right)}t\right)\leq e^{-\beta_{d}n\text{Var}\left(Z\right)t^{2}},\label{eq:11}
\end{alignat}
where $\beta_{d}$ is a positive constant depending only on $d$ and
$C$.

We infer that for all $C>0$, all $n$ large enough (depending on
$d$ and $C$) and any integer function $k\left(n\right)$ satisfying
$k\left(n\right)\geq C\log$$\left(n\right)$
\begin{alignat}{1}
\mathbb{P}\left(\left\Vert H\right\Vert _{S_{2k\left(n\right)}}\geq\Phi\left(\sigma_{k\left(n\right)},\sigma_{k\left(n\right)}^{*}\right)+t\right) & \overset{\left(1\right)}{\leq}\mathbb{P}\left(\left\Vert H\right\Vert _{S_{2k\left(n\right)}}\geq\sqrt[2k\left(n\right)]{\mathbb{E}\left[\left\Vert H\right\Vert _{S_{2k\left(n\right)}}^{2k\left(n\right)}\right]}+t\right)\nonumber \\
 & \overset{\left(2\right)}{\leq}\mathbb{P}\left(\left\Vert H\right\Vert _{S_{2k\left(n\right)}}\geq\mathbb{E}\left[\left\Vert H\right\Vert _{S_{2k\left(n\right)}}\right]+t\right)\nonumber \\
 & \overset{\left(3\right)}{\leq}e^{-\beta_{d}n\text{Var}\left(Z\right)t^{2}},\label{eq:22}
\end{alignat}
where $\left(1\right)$ follows from equation \eqref{eq:4.25}, $\left(2\right)$
follows from Jensen's inequality and $\left(3\right)$ follows from
equation \eqref{eq:11}.

Fix $\varepsilon>0$. The proof of Corollary \ref{cor:Norm_of_H},
along with the assumption $n\text{Var}\left(Z\right)\gg\log\left(n\right)$
imply $\limsup_{n\to\infty}\Phi\left(\theta_{k_{0}},\theta_{k_{0}}^{*}\right)\leq2\sqrt{d}$,
for $k_{0}:=k_{0}\left(n\right)=\left\lceil \sqrt{n\text{Var}\left(Z\right)\log\left(n\right)}\right\rceil $.
Therefore there exists $N_{\varepsilon}\in\mathbb{N}$ such that $\Phi\left(\theta_{k_{0}},\theta_{k_{0}}^{*}\right)<2\sqrt{d}+\frac{1}{2}\varepsilon$,
for all $n>N_{\epsilon}$. The last bound, together with the upper
bound in \eqref{eq:22} and the fact that $\left\Vert H\right\Vert _{2}\leq\left\Vert H\right\Vert _{S_{2k}}$
for any $k\in\mathbb{N}$, imply that for all large enough $n$ (depending
on $\varepsilon$ and $d$ )
\begin{alignat}{1}
\text{ }\mathbb{P}\left(\left\Vert H\right\Vert _{2}\geq2\sqrt{d}+\varepsilon\right) & \leq\mathbb{P}\left(\left\Vert H\right\Vert _{S_{2k_{0}}}\geq2\sqrt{d}+\varepsilon\right)\nonumber \\
 & \leq\mathbb{P}\left(\left\Vert H\right\Vert _{S_{2k_{0}}}\geq\Phi\left(\theta_{k_{0}},\theta_{k_{0}}^{*}\right)+\frac{1}{2}\varepsilon\right)\nonumber \\
 & \leq e^{-\frac{1}{4}\beta_{d}n\text{Var}\left(Z\right)\varepsilon^{2}}.\label{eq:23}
\end{alignat}
Our postulation $n\text{Var}\left(Z\right)\gg\log\left(n\right)$
implies that all large enough $n$ (depending on $\beta_{d}$ and
$\varepsilon$ ) obeys
\[
\frac{8}{\beta_{d}\varepsilon^{2}}\log\left(n\right)<n\text{Var}\left(Z\right),
\]
and hence 
\[
\mathbb{P}\left(\left\Vert H\right\Vert _{2}\geq2\sqrt{d}+\varepsilon\right)\leq n^{-2}.
\]
By the Borel-Cantelli we have almost surely $\limsup_{n\to\infty}\left\Vert H\right\Vert _{2}\leq2\sqrt{d}+\epsilon$.
Since $\epsilon$ was arbitrary, we conclude that $\limsup_{n\to\infty}\left\Vert H\right\Vert _{2}\leq2\sqrt{d}$
almost surely. Since $\liminf_{n\to\infty}\left\Vert H\right\Vert \geq2\sqrt{d}$
almost surely (see \cite[Remark 5.2]{RK17}) it thus follows that
$\lim_{n\to\infty}\left\Vert H\right\Vert _{2}=2\sqrt{d}$ almost
surely, which concludes the proof of Corollary \ref{cor:Norm_of_H_2}.
\hfill\qed

\subsection{Proof of Theorem \ref{thm:eigenvalue_confinement}}

One can observe in the proof of \cite[Theorem 6.1]{RK17}, that the
result is valid whenever one replaces the function $\mathcal{E}\left(\xi\right)$
with any upper bound on $\mathbb{P}\left(\left\Vert H\right\Vert >2\sqrt{d}+\xi\right)$.
Using Corollary \ref{cor:Norm_of_H_2} together with the last observation,
we infer 
\begin{cor}
\label{cor:5.1}Assume $d\geq2$ and $nq\gg\log\left(n\right)$, then
for all large enough $n$ (depending on $d$)
\begin{enumerate}
\item For every $\xi>0$, the ${n-1 \choose d}$ smallest eigenvalues of
the matrix $A$ are within the interval $-pd+\sqrt{nq}\left[-2\sqrt{d}-\xi,2\sqrt{d}+\xi\right]$
with probability at least $1-e^{-\frac{1}{4}\beta_{d}nq\xi^{2}}$.
\item For every $\xi>0$ and $\xi'>0$, if $nq\geq\frac{d\left(2d+2\xi'\right)^{6}\log^{6}\left(n\right)}{n}$,
then the remaining ${n-1 \choose d-1}$ eigenvalues of $A$ are inside
the interval $nq+\left[-\Gamma\left(\xi,\xi',n\right),\Gamma\left(\xi,\xi',n\right)\right]$
with probability at least $1-e^{-\frac{1}{4}\beta_{d}nq\xi^{2}}-\mathscr{E}\left(\xi'\right)$,
where
\[
\Gamma\left(\xi,\xi',n\right)=pd+\frac{\left(2\sqrt{d}+\xi\right)^{2}\sqrt{nq}}{\sqrt{nq}-4\left(2\sqrt{d}+\xi\right)}+100d^{\frac{7}{2}}\left(d+\xi'\right)^{3}\sqrt{q}\log^{3}\left(n\right),
\]
and
\[
\mathscr{E}\left(\xi\right)=\frac{4e^{3}d^{\frac{5}{2}}}{\left(d-1\right)!}\exp\left(5\log\left(2d+2\xi\right)+5\log\left(\log\left(n\right)\right)-\xi\log\left(n\right)\right).
\]
\end{enumerate}
\end{cor}

Note that if $nq\gg\log\left(n\right)$ then for all $D>0$ and all
$\xi>0$ we have for all large enough $n$ (depending only on $D$,
$\xi$ and $d$)
\begin{equation}
nq\geq\frac{4D}{C_{d}\left(\sqrt{d}\frac{\xi}{2}\right)^{2}}\log\left(n\right),\label{eq:1-2}
\end{equation}
Moreover, since $nq\gg\log\left(n\right)$ we have for all large enough
$n$ (depending on $d$ and $\xi$)
\[
pd\leq\frac{1}{4}n\left(1-p\right)\xi^{2},
\]
thus 
\begin{equation}
-pd+\sqrt{dnq}\left[-2-\frac{\xi}{2},2+\frac{\xi}{2}\right]\subseteq\sqrt{dnq}\left[-2-\xi,2+\xi\right].\label{eq:3-1}
\end{equation}
Using the first part of the following Corollary, together with \eqref{eq:1-2}
and \eqref{eq:3-1}, we infer that for every $D>0$, every $\xi>0$
and all large enough $n$ (depending on $D$, $\xi$ and $d$), the
${n-1 \choose d}$ smallest eigenvalues of the matrix $A$ are within
the interval $\sqrt{dnq}\left[-2-\xi,2+\xi\right]$ with probability
at least $1-n^{-D}.$

Turning to confine the remaining eigenvalues, using \cite[Theorem 6.1]{RK17},
it follows that $\mathscr{E}\left(\xi_{D}\right)\leq\frac{n^{-D}}{2}$
for an appropriate choice of $\xi_{D}=C'(D+1)>0$, with $C'$ depending
only on $d$. Recalling our assumption $nq\gg\log\left(n\right)$,
which implies $nq\geq\log\left(n\right)$ for all large enough $n$,
it follows that for all large enough $n$ (depending on $D$ and $d$),
$nq\geq\frac{d\left(2d+2\xi{}_{D}\right)^{6}\log^{6}\left(n\right)}{n}$.
We may therefore apply part $\left(2\right)$ of Corollary \ref{cor:5.1},
which together with \eqref{eq:1-2}, shows that for all large enough
$n$ (depending on $D$, $\xi$ and $d$), the remaining ${n-1 \choose d-1}$
eigenvalues of $A$ are inside the interval $nq+\left[-\Gamma\left(\xi,\xi{}_{D},n\right),\Gamma\left(\xi,\xi{}_{D},n\right)\right]$
with probability at least $1-n^{-D}.$ 

Taking $\xi_{0}:=\left(\sqrt{5}-2\right)\sqrt{d}$ and $\xi'=\xi_{D}$,
and can verify that 
\[
\Gamma\left(\xi_{0},\xi{}_{D},n\right)\leq\frac{13}{2}d+100d^{\frac{7}{2}}\left(d+C'(D+1)\right)^{3}\sqrt{q}\log^{3}\left(n\right).
\]
We may therefore conclude that for any $D>0$ and all large enough
$n$ (depending on $d$), the remaining ${n-1 \choose d-1}$ eigenvalues
of $A$ are inside the interval
\[
nq+\Big[\frac{13}{2}d+100d^{\frac{7}{2}}\left(d+C'(D+1)\right)^{3}\sqrt{q}\log^{3}\left(n\right)\Big]\cdot[-1,1]
\]
with probability at least $1-n^{-D}.$ Since by assumption $q\log^{6}\left(n\right)\leq\frac{1}{C\left(1+D\right)^{6}}$,
it follows that the eigenvalues are within the interval 
\[
nq+7d\cdot[-1,1],
\]
provided $C>0$ is chosen large enough (depending only on $d$).\hfill\qed 

\section{Discussion and open questions }

The study of spectrum and in particular the spectral gap raises many
open questions. 
\begin{itemize}
\item Theorem \ref{thm:main_result} and Corollary \ref{cor:Spectral_gap}
provides bounds on the spectral gap in the regime $nq\gg\log n$.
In \cite{MR4116720,MR3945756} it was shown that for the Erd\H{o}s--Rényi
model, namely the case $d=1$, this is the optimal regime. Can one
prove a similar result in the case $d\geq2$?
\item An interesting question that can be asked is regarding the asymptotic
behavior of $\left\Vert H\right\Vert _{S_{2k}}$ and $\left\Vert H\right\Vert _{2}$
in the regime where $d:=d\left(n\right)$. Namely, to consider the
random simplicial complex $X\left(d,n,p\right)$, under the assumption
that $\dim\left(X\right)$ is a function of $n$. 
\item The tail bound obtained in Corollary \ref{cor:Norm_of_H_2} is not
likely to be optimal in the power of $\epsilon$. One can wonder what
is the best power of $\epsilon$ that can be achieved. A more difficult
question is to prove the existence of a limiting distribution $\Psi(t)=\lim_{n\to\infty}\mathbb{P}(\left\Vert H\right\Vert _{2}\geq2\sqrt{d}+t)$
and calculate it. 
\item Our work focuses on a specific model of random simplicial complexes,
namely the Linial-Meshulam model. One can consider different models
of random simplicial complexes, and perhaps use similar tools in order
to establish analogous results. One example of such model is the multi-parameter
random simplicial complex model \cite{CF16,MR3959923}. Another interesting
model to study is the high dimensional analogue of random regular
graph called random Steiner systems, see \cite{MR4027617,RT20}.
\end{itemize}

\bibliographystyle{alpha}
\phantomsection\addcontentsline{toc}{section}{\refname}\bibliography{Bibliography}

$ $\\
Department of mathematics,\\
Technion - Israel Institute of Technology\\
Haifa, 3200003, Israel.\\
Email: leibzirers@gmail.com\\
Email: ron.ro@technion.ac.il
\end{document}